\documentclass[a4paper,11pt]{amsart}
\usepackage{amsmath,amsthm,amsfonts,shuffle,amssymb}
\usepackage{fullpage,comment,graphicx,epstopdf}
\usepackage{multicol}
\usepackage{multirow}
\usepackage{caption}
\usepackage{amsmath,graphicx,fancyhdr,array,rotating,verbatim,amssymb,version}
\usepackage[all,matrix,arrow,import,arc,color]{xy}

\usepackage[inline]{enumitem}
\setlist[enumerate,1]{label=\textit{\alph*)}}

\usepackage{algorithm,algpseudocode}

\input{xy}
\xyoption{all}
\xyoption{matrix}

\usepackage{tabularray}

\usepackage{color,graphicx,times,xcolor}

\theoremstyle{plain}
\newtheorem{theorem}{Theorem}
    \newtheorem{lemma}[theorem]{Lemma}
  \theoremstyle{remark}
    \newtheorem{remark}[theorem]{Remark}
 
\theoremstyle{definition}
  \newtheorem{example}[theorem]{Example}
  \theoremstyle{definition}
  \newtheorem{definition}[theorem]{Definition}
  \theoremstyle{plain}
    \newtheorem{proposition}[theorem]{Proposition}
  \theoremstyle{plain}
    \newtheorem{corollary}[theorem]{Corollary}

\DeclareMathOperator{\indeg}{indeg}

\def\neg{\mathrm{neg}}

\def\N{\mathbb{N}}

\def\Z{\mathbb{Z}}
\def\R{\mathbb{R}}
\def\Q{\mathbb{Q}}

\def\m1{^{ \hbox{\small{-}}1}}

\usepackage{todonotes}

\usepackage{hyperref}

\title[Beyond Boolean networks]{Beyond Boolean networks: new tools for the steady state analysis of
multivalued networks}

\author[Garc\'{\i}a Galofre et al.]{J. Garc\'{\i}a Galofre${}^1$, M. P\'erez Mill\'an${}^{2,3}$, A. Galarza Rial${}^4$, \\ R. Laubenbacher${}^5$, A. Dickenstein${}^{2,3}$}

\address{$^1$ Departamento de Matem\'atica y Ciencias, Universidad de San Andres, 
(B1644BID) Victoria, Prov. B. Aires, Argentina.}

\address{$^2$ Instituto de Investigaciones Matem\'aticas ``Luis A. Santal\'o" (UBA-CONICET), C1428EGA Buenos Aires, Argentina.}

\address{$^3$ Departamento  de Matem\'atica, FCEN, Universidad de Buenos Aires,
    Ciudad Universitaria, Pab.\ I, C1428EGA Buenos Aires, Argentina.}

\address{$^4$ Departamento de Ciencias Exactas, Ciclo B\'asico Com\'un, Universidad de Buenos Aires,
 C1405CAE Buenos Aires, Argentina}

\address{$^5$ Laboratory for Systems Medicine, Department of Medicine, University of Florida,
Gainesville, FL, USA.}

\email{jgarciagalofre@udesa.edu.ar, mpmillan@dm.uba.ar, ayegari89@gmail.com,\newline reinhard.laubenbacher@medicine.ufl.edu, alidick@dm.uba.ar}
\date{}

\thanks{MPM and AD were partially supported by UBACYT 20020220200166BA and CONICET PIP 11220200100182CO, Argentina. RL was partially supported by NIH Grants 1 R01 HL169974-01 and 1U01EB024501-01.} 

\begin{document}

\begin{abstract}

Boolean networks can be viewed as functions on the set of binary strings of a given length, described via logical rules. They were introduced as dynamic models into biology, in particular as logical models of intracellular regulatory networks involving genes, proteins, and metabolites. Since genes can have several modes of action depending on their expression levels, binary variables are often not sufficiently rich, requiring the use of multivalued
networks instead. 

In this paper, we explore the multivalued generalization of Boolean networks by writing the standard $(\wedge, \vee, \lnot)$ operations on $\{0, 1\}$ in terms of the operations $(\odot, \oplus, \neg)$ on $\big\{0,\frac{1}{m}, \frac{2}{m}, \dots, \frac{m-1}{m}, 1\big\}$ from multivalued logic.
We recall the basic theory of this mathematical framework, and give a novel algorithm for computing
the fixed points that in many cases has essentially the same complexity as in the binary case.
 Our approach provides a biologically intuitive representation of the network.  
    Furthermore, it uses tools to compute lattice points in rational polytopes, tapping a rich area of algebraic combinatorics as a source for combinatorial algorithms for  network analysis. An implementation of the algorithm is provided.  
\end{abstract}

 \maketitle

\section{Introduction}

Boolean networks have been used in computer science, engineering, and physics extensively. They were introduced into biology by Stuart Kauffman~\cite{Kauffman69} in 1969, as a model class for intracellular regulatory networks, viewed as logical switching networks rather than as biochemical reaction networks. This has proven very useful, in part because their specification is intuitive from a biological point of view, and there are now hundreds of published Boolean network models of this type.  

The setting of Boolean networks owes much to the work of the Belgian scientist Ren\'e Thomas. His research included DNA biochemistry,   genetics, biophysics, mathematical biology, and dynamical systems. His purpose was to decipher the key logical principles at the basis of the behavior of biological systems and how complex dynamical behavior could emerge. He realized that the researcher's intuition was not enough to understand the intrincated biological regulatory networks and he then proposed to formalize the models using the operations in Boolean algebra, where variables take only two values ($0$ or OFF and $1$ or ON) and the logical operators are AND, OR and NOT, starting with~\cite{Thomas73} in 1973. In the article~\cite{Thomas91} published in 1991 he argues that even if one is interested in a logical description, the binary case could be too simple in many applications. He adds that for instance, when there is a variable with more than a single action, different levels might be required to produce different effects.  For applications of this modeling framework to biology, it is thus often useful to be able to assume that certain variables have more than two possible values (see for instance~\cite{Albert}). In gene regulation, for instance, a given gene may have several models of action, depending on its expression levels. In order to capture this complexity, a variable might need to take several different values, such as ``0 (off), 1 (medium expression), 2 (high expression)." The functions governing such {\em multivalued } networks can be expressed in terms of multivalued logic. The analysis of {\em steady states} (also known as {\em stable states} or {\em fixed points}) becomes significantly more complex. 
This paper presents an algorithm to carry out such analysis in a computationally new way. 
Before expanding on our results, we first summarize other approaches to this question.

For models with a small number of variables, exhaustive model simulation is feasible, while for larger models, a method used in the Boolean setting is sampling of the state space of the model.  A first idea to treat a multivalued model is to translate it into a  Boolean network~\cite{Ham}. More variables are needed and the 
Boolean network is in principle only partially defined (see for instance~\cite{Tonello} and the references therein).

The logical formalism, originally used to model regulatory and signaling Boolean networks, has been extended to allow for 
multiple values of the different nodes. 
Several groups contributed various methods and tools for the analysis of multivalued logical models. 
In the 2003 article~\cite{Devloo}, image functions are introduced to compute the value of the logical parameter associated with a logical variable depending on the state of the system; they use constraint programming, a method for solving constraint  satisfaction problems. In the article~\cite{Schaub} from 2007, qualitative (multivalued) networks were introduced, using formal verification methods to check the consistency of a model with  respect to experimental data and to analyze the steady states.
In the 2007 paper~\cite{Naldi07}, the authors propose a method to compute steady states of Boolean and  {\em unitary}  multivalued models (see the paragraph before Definition~\ref{def:cont}) using decision diagrams. This approach is implemented in the Gene Interaction Network simulation suite (GINsim), a software designed for the qualitative modeling and analysis of regulatory networks~\cite{Naldi09} using rule-based descriptions. Some recent methodological advances of this logic framework are illustrated in~\cite{Chaouiya24}. The BioModelAnalyzer~\cite{BMA} presented in 2018 is a visual tool for modeling and analyzing biological networks based upon formal methods developed for the specification and verification of properties in concurrent software systems.
The CoLoMoTo Interactive Notebook~\cite{CoLoMoTo} provides an interesting and  unified environment to edit, execute, share, and reproduce analyses of qualitative models of biological networks since 2018; it relies on Docker and Jupyter technologies. The recent paper~\cite{Trap} generalizes trap spaces of Boolean networks to the multivalued setting, providing a Python implementation. 
A different approach to model multivalued networks as logic programs since 2014 utilizes the formulation of fuzzy logic~\cite{FAST, FAST1}. They integrate the operations from multivalued logic that we describe in Section~~\ref{sec:prelim} with the paradigms of Fuzzy Answer Set Programming (FASP) to model the dynamics of Gene Regulatory Networks. FASP is an extension of the Answer Set Programming (ASP), a paradigm that allows for modeling and solving combinatorial search problems in continuous domains.
These multivalued logical modelings admit a different number of values for each node, but a precise complexity computation is lacking in general. 

When the number of values is the same for all nodes  and there are $n$ nodes or variables, one possible algebraic approach is the following. Denote the number of values by the integer $m+1$ (so, the Boolean case corresponds to $m=1$). In case $m+1$ is a power of a prime number, any function on $n$ variables defined over  a finite set $X$ of cardinality $m+1$ can be thought of as a function over a finite field. Moreover, it can be expressed as a polynomial function with coefficients in this field, and this opens the use of tools from computational  algebraic geometry such as Gr\"obner bases~\cite{VCJL}. The advantage is that it is not necessary to express the function via a whole exponentially large table.  
For instance, given a function $f:X^ n \to X^n$, we can find the {\em fixed points} $\{ x \in X^n \, : \, f(x)=x\}$ 
 of $f$  by solving the square polynomial system $ f(x) -x =0$.  In~\cite{BFSS13} the authors show that using this approach in the Boolean case, it possible to compute the fixed points with complexity $O(2^{0.841 n})$ (under some assumptions), while the bound for exhaustive search is $ 4 \,  {\rm log}_2(n) \,  2^n$. A different algebraic approach is considered in~\cite{LC10}, where the authors study multilinear functions representing network functions $f:X^n \to X^n$.  

Our starting point is the following.
Given $n$ interacting (biological) species with $m+1$ possible values each, we can establish a bijection between the set of values and the set
\begin{equation}\label{eq :X}
X_m=\big\{0,\tfrac{1}{m}, \tfrac{2}{m}, \dots, \tfrac{m-1}{m}, 1\big\} \subset [0,1].
\end{equation}
Our aim is to study the dynamics of the iteration of a function $f: X_m^n \to X_m^n$. We  use in this multivalued setting the operations $\odot$,  $\oplus$ and $\neg$  from multivalued logic~\cite{Cignoli,  Epstein} (introduced in Section~\ref{sec:prelim}), which coincide with the standard Boolean operations AND, OR and NOT when $m=1$. A standard notation for $X_m$ in the setting of MV-algebras is $L_{m+1}$, where $L$ stands for {\L}ukasiewicz. 

These operations are more intuitive and closer to biological interpretations than the usual  operations of polynomials over a finite field (and there is no restriction on the number of values). See for instance Figure~\ref{fig:polyn_vs_multivalued}, that features an example of two interacting nodes, with three values each, where the second variable acts as a repressor of the first one.  The update function  $f_1$ of the first node is represented both as a polynomial over the field with $3$ elements, which gives no clue of its behavior, and in terms of the logic operations we propose, whose interpretation is transparent. We also give the input in terms of target functions as in  GINsim, which is less compact (and the formulae grow rapidly as the number of values increases).
\begin{figure}[ht!]
    \centering
\begin{minipage}{.4\textwidth}
\centering
    \renewcommand{\arraystretch}{1.2}
     \makebox[.5\linewidth]{$
    \begin{array}{c||c|c|c}
    x_1\setminus x_2 &0&\tfrac{1}{2}&1\\
    \hline
    \hline
    0&0&0&0\\
    \hline
    \tfrac{1}{2}&\tfrac{1}{2}&0&0\\
    \hline
    1&1&\tfrac{1}{2}&0\\
    \end{array}
    $}
    
    \medskip
    
    \makebox[.5\linewidth]{(a) $f_1:\{0,\tfrac{1}{2},1\}^2\to\{0,\tfrac{1}{2},1\}$}
    \makebox[.5\linewidth]{$f_1(x_1,x_2)=x_1 \ominus x_2 $}

\end{minipage}
 \begin{minipage}{.5\textwidth}
 \centering
    \renewcommand{\arraystretch}{1.2}
    \makebox[.5\linewidth]{
    $\begin{array}{c||c|c|c}
    x_1\setminus x_2 &0&1&2\\
    \hline
    \hline
    0&0&0&0\\
    \hline
    1&1&0&0\\
    \hline
    2&2&1&0\\
    \end{array}$
    }

    \medskip
    
    \makebox[.5\linewidth]{(b) $f_1:\mathbb{F}_3^2\to\mathbb{F}_3 $}
    \makebox[.5\linewidth]{$f_1(x_1,x_2)=x_1 +2 x_1x_2 +x_1^2x_2+x_1x_2^2+x_1^2x_2^2$}
\end{minipage}

\medskip

    \makebox[.5\linewidth]{(c) $f_1:\{0,1,2\}^2\to\{0,1,2\} $}
    \makebox[.5\linewidth]{$f_1(x_1,x_2)=\begin{cases}
        2  & \text{if }(x_1{:}2~\&~ !x_2)\\
        1  & \text{if }(x_1{:}1 ~\&~ !x_2)|(x_1{:}2 ~\&~ x_2{:}1)\\
        0 & \text{otherwise}.
    \end{cases}$}
    \caption{ (a) The nodes take the values $\{0,\tfrac{1}{2},1\}$. The repression shown in the table is naturally described by the function $f_1=x_1 \ominus x_2 = x_1\odot\neg(x_2) = \max \{0, x_1-x_2\}$. In (b) and (c)  the nodes take the values $\{0,1,2\}$ and we consider the translated table. In (b), the polynomial $f_1\in\mathbb{F}_3[x_1,x_2]$ summarizes the values in the table over the field $\mathbb F_3$. In (c), the table is instead expressed by the logical function $f_1$, where the notation means $\&$ $\leftrightarrow$ AND, $|$ $\leftrightarrow$ OR, $x_j{:}i$ $\leftrightarrow$ $x_j=i$ and $!x_j$ $\leftrightarrow$ $x_j=0$.}
    \label{fig:polyn_vs_multivalued}
\end{figure}

The study of Boolean networks could have high complexity, even if they are composed of simple elementary units~\cite{R23}.
We show that via the proposed operations of multivalued logic, which are indeed expressed in terms of  linear inequalities, we can recover most of the good properties of Boolean networks. Moreover,  we show that the computation of fixed points can be done algorithmically for any $m$ using tools to find lattice points in rational polytopes  and that in many cases, the complexity to compute the fixed points is essentially the same as in the Boolean setting.  As we mentioned, these operations were already considered in the papers~\cite{FAST, FAST1} from 2014 and 2018, where $X_m$ is denoted by $\Q_m$, but we introduce a novel computational approach. Even if we work with a synchronous update of the nodes, we present in Example~\ref{ex:TD}  a small network extracted from the foundational book~\cite{Thomas_DAri}, where the synchronous multivalued setting can reproduce the features of a model with time delays, as the networks become more {\it expressive}.

 \smallskip

This article has two parts: the first one, consisting of Sections~\ref{sec:prelim} and~\ref{sec:dotneg}, introduces 
 functions, structure, and motifs of multivalued  networks; the second one, consisting of Sections~\ref{sec:reductions},~\ref{sec:sstates} and~\ref{sec:bioexample}, presents theoretical and algorithmic results to compute the fixed points, together with biological examples.

In Section~\ref{sec:prelim} we introduce the logic operations $\oplus,\odot,\neg$ and highlight their main properties. We give, in Theorem~\ref{th:logic}, a constructive way of turning data from a table into an expression of the function in terms of $\odot, \neg$ and constant functions.   When modeling biological networks, it is useful to have in mind the behavior of different small networks that appear frequently as subnetworks of bigger ones, called {\em motifs}. We show  the behavior of several small motifs.  Our aim is to present simple examples of mild activators/inhibitors that a biologist could choose to try to match the experimental data.

In Section~\ref{sec:dotneg}, we introduce in Definition~\ref{def:dotneg} the notion of $\odot$-$\neg$ functions.   Networks defined by $\odot$-$\neg$ functions can be directly read from  their associated wiring diagrams. Algorithm~1 shows how to translate any network function into a $\odot$-$\neg$ function adding new variables and we show in Theorem~\ref{thm:DN} how to recover the fixed points of the original network.

In Section~\ref{sec:reductions} we propose several possible reductions of the number of variables (inspired by~\cite{VCLA} in the Boolean setting) without increasing the indegree, that is, without increasing the maximum number of variables on which each variable depends (see  Proposition~\ref{prop:indegree}). We present in Section~\ref{ssec:mammalian} an example 
which is the translation to our setting of an interesting network extracted from the website \url{cellcolective.org}, where the number of nodes  is significantly decreased by the reduction process. In the most interesting cases, the computation is too heavy to be done by hand and we used our publicly available implementation~\cite{Aye}  in Python language, that we explain in the next section.

In Section~\ref{sec:sstates}, we address the computation of fixed points for $\odot$-$\neg$ networks. As we mentioned, this gives in turn  the fixed points of any network function $F: X_m^n \to X_m^n$ by Theorem~\ref{thm:DN}. One crucial algorithmic fact for $\odot$-$\neg$ networks is Theorem~\ref{thm:ell}  that allows us to automatize the computations without simulations, which would be too costly in general. We present the basic setting for our implementation in Algorithm~\ref{algo2} and we discuss the cost of these computations. When the assumptions of Theorem~\ref{th:regions} do not hold, one needs to count the number of lattice points in a rational polytope, that is, all the integer solutions of a system of inequalities given by linear forms with integer coefficients defining a bounded region.  Barvinok proposed in~\cite{Barvinok} an algorithm
that counts these lattice points in polynomial time when the dimension of the polytope is fixed. The theoretical basis was given by Brion in the beautiful paper~\cite{Brion}, where he uses the relation of rational polytopes with the theory of toric varieties and results in equivariant $K$-theory.  
We exemplify the use of our free app for the computation of fixed points  in the network introduced in Section~\ref{ssec:mammalian}. The computation of fixed points in the Boolean case becomes straightforward in this setting, as it is reduced to just checking some set containments (see Lemma~\ref{lem:algoBool}).

In Section~\ref{sec:bioexample} we present a multivalued model for the denitrification network in \textit{Pseudomonas aeruginosa} based on the discrete-time and multivalued deterministic model introduced in~\cite{ABL}. \textit{P. aeruginosa} can perform (complete) denitrification, a respiratory process that eventually produces atmospheric N$_2$. The external parameters in the model and some variables are treated as Boolean (low or high), and other variables are considered ternary (low, medium or high). The update rules are formulated in ~\cite{ABL} in terms of MIN, MAX and NOT (which correspond to AND, OR and NOT in a Boolean setting). However, the regulations of a few variables cannot be expressed in terms of these operators (marked with a * in the column ``Update Rules" in  their Table~3). As stated in Theorem~\ref{th:logic}, any function can be written in terms of constants and the logical operations  $\oplus, \odot, \neg$ we propose to use in this paper.  We can thus exactly represent their update rules and find the fixed points using our techniques.

\section{The operations of MV-algebras and the representation of functions}\label{sec:prelim}

This section introduces the basic operations in multivalued logic and how we use them to represent the functions in biological models.

\smallskip

We fix a natural number $m$ and  consider networks with $m+1$ values in the set 
\begin{equation*}\label{eq:M}
X_m=\big\{0,\tfrac{1}{m}, \tfrac{2}{m}, \dots, \tfrac{m-1}{m}, 1\big\}\subset [0,1 ].
\end{equation*}

Boolean networks are a particular case when $m=1$. As $X_m$ is a subset of the real numbers $\R$, we can consider the standard operations of addition/substraction ($+/-$) and multiplication ($\cdot$)  on the real line. However, we introduce in the multivariate setting the operations $\odot$,  $\oplus$ that come from multivalued logic~\cite{Cignoli, Epstein}, which are more intuitive and closer to biological interpretations.

\subsection{The MV operations}\label{ssec:MV}

The three basic operations in multivalued logic that we use to build all the functions that we consider are introduced in the following definition:

\begin{definition} \label{def:opers}
We consider the following operations/maps on $X_m$:
 \begin{align}
 \label{neg}  \neg:X_m\rightarrow X_m, \text{ where } \neg(x) = 1-x. \\
 \label{oplus} \oplus:X_m\times X_m\rightarrow X_m, \text{ where } x\oplus y=\min\{1, x  \boldsymbol{+}y\}. \\
 \label{odot} \odot:X_m\times X_m\rightarrow X_m, \text{ where } x\odot y=\max\{0,x+y-1\}.
 \end{align}
\end{definition}

Note that these operations can be defined over the interval $[0,1]$ and then restricted to $X_m$ for any value of $m$ since their expressions do not depend on $m$.

\begin{example}
 When $m=2$, the operations \eqref{oplus} and \eqref{odot} in chart presentation are as follows:
  \renewcommand{\arraystretch}{1.2}
  \[
\begin{array}{c||c|c|c}
\oplus&0&\tfrac{1}{2}&1\\
\hline
\hline
0&0&\tfrac{1}{2}&1\\
\hline
\tfrac{1}{2}&\tfrac{1}{2}&1&1\\
\hline
1&1&1&1\\
\end{array}
\hskip 1cm
\begin{array}{c||c|c|c}
\odot&0&\tfrac{1}{2}&1\\
\hline
\hline
0&0&0&0\\
\hline
\tfrac{1}{2}&0&0&\tfrac{1}{2}\\
\hline
1&0&\tfrac{1}{2}&1\\
\end{array}
\]
\end{example}

\begin{remark}\label{rem:multi_prod}
The following relations between $\oplus$ and $\odot$, which are formally equal to the De Morgan's law between AND and OR in the Boolean case, hold:
 \begin{align}
  \label{eq:odot_as_oplus} x\odot y=\neg\left(\neg(x)\oplus \neg(y)\right),\\
  \label{eq:oplus_as_odot} x\oplus y= \neg\left(\neg(x)\odot \neg(y)\right).
 \end{align}
\end{remark}

We list in the following proposition other useful properties which are straightforward to prove.

\begin{proposition}\label{prop:oper}
Given $m$ and $X_m$ as in~\eqref{eq:M}, consider the operations defined in Definition~\ref{def:opers}. The following properties hold:
\begin{enumerate}[label=(\roman*)]
\item $\neg(0)=1$ and $\neg(1)=0$.
\item  The operations $\oplus$ and $\odot$ are associative and commutative. 
 \item \label{it:several_products} The product of several variables can be expressed as 
\begin{equation*}\label{eq:prodn}
x_1\odot \dots \odot x_n=\max\{0,x_1+\dots+x_n-(n-1)\}.
\end{equation*}
\item  \label{it:several_sums} Analogously, 
$$x_1\oplus \dots \oplus x_n=\min\{1,x_1+\dots+x_n\}.$$
\item   $0\odot x=0$ and $1\odot x=x$ for any $x \in X$. 
\item For any finite index set $I$, $\neg\left(\bigoplus_{i\in I} a_i\right)=\bigodot _{i\in I}\neg(a_i)$.
\item For $x,y \in X$, we have the equalities 

$\max\{x,y\} \, = \, \left(x \odot \neg(y)\right) \oplus y$,  \quad
$\min\{x,y\} = \, (x \oplus \neg(y)) \odot y$. 
\item  $x \le y$ if and only if $x \odot \neg(y) =0$.
\item $x \odot y =0$ if and only if $x+y \le 1$. More generally, $x_1\odot \dots \odot x_n =0$ if and only if $x_1 + \dots + x_n \le (n-1)$.
\end{enumerate}
\end{proposition}

\begin{remark}\label{rem:nondist}
Associativity will allow us to avoid writing parentheses in some cases. More precisely, we will write $y_1\odot y_2 \odot y_3$ instead of $\big(y_1\odot y_2\big)\odot y_3\hspace{4mm} \hbox{or}\hspace{4mm} y_1\odot \big(y_2\odot y_3\big)$ where $y_i=x_i$ or $y_i=\neg(x_i)$.

On the other hand, the operations we defined do not satisfy distributivity. In general, $(x\oplus y)\odot z\neq (x\odot z)\oplus (y\odot z)$. Consider for example  $m=5, x=\frac{3}{5}, y=\frac{2}{5}, z=\frac{1}{5}$ then $(x\oplus y)\odot z=\frac{1}{5}$ while $(x\odot z)\oplus(y\odot z)=0.$
\end{remark}
 
 \begin{definition} \label{def:minus}
  We also define substraction, exponentiation and multiplication by a positive integer $k\in \N$ as:
  \begin{align} 
  \label{eq:minus}
   x\ominus y &:= x\odot \neg(y)=\max\{0,x-y\}=\left\{ \begin{array}{lcc}
             x-y &   \text{if}  & x-y \geq 0 
             \\ 0 &  \text{otherwise}  & 
             \end{array}
   \right.,\\
   \label{eq:exp} 
   x^k &:= \underset{k \text{ times}}{\underbrace{x\odot \dots \odot x}} =\max\{0,k x - (k-1)\}, \\
   \label{eq:scalar_mult}
    kx &:= \underset{k \text{ times}}{\underbrace{x\oplus \dots \oplus x}}= \min\{1,k x\}.
  \end{align}
 \end{definition}
 
\begin{remark} \label{rem:menor}
    Note that $x \odot y \le x$ for any  $x,y \in X_m$, and equality holds if and only if $y=1$ or $x=0$.  \newline In particular, $x^k \le x$ $\forall x \in X_m$ and  $\forall k \in \N$. The equality holds if and only if $x\in \{0, 1\}$ or $k=1$.
\end{remark}

\begin{remark} \label{rem:negkx}
    It is interesting to note that $\neg(kx)=\neg(x)^k$, indeed:
    \[
    \neg(kx)=\neg(\underset{k \text{ times}}{\underbrace{x\oplus \dots \oplus x}})=\neg\big(\neg(\underset{k \text{ times}}{\underbrace{\neg(x)\odot \dots \odot \neg(x)}} )\big)=\neg(x)^k.
    \]
    
\end{remark}

\subsection{Representability of functions} \label{ssec:representation}
We show next how to represent a function $f: X_m^n \to X_m$, in terms of $\odot$ and $\neg$ operators in a constructive way. We start with an example.

\begin{example}\label{ex:f0f1}
Let $m=2$, so that $X_m=X_2=\left\{0,\frac{1}{2}, 1\right\}$.
Define the following functions
\[f_0(x)=\neg(x)^2, \; f_1(x)=x^2, \text{ and } f_{\frac{1}{2}}(x)=\neg(f_0(x))\odot \neg(f_1(x)).\]

Then, 
 \[
 \begin{array}{c|c|c|c}
  x&f_0(x)&f_{\frac{1}{2}}(x)&f_1(x)\\
  \hline
  0&1&0&0\\
  \frac{1}{2}&0&1&0\\
  1&0&0&1
 \end{array}
 \]

\smallskip
 
 This means that  $f_0, f_{\tfrac{1}{2}}, f_1$ are indicators of the points in $X_2$ and they allow us to write any function $f:X_2 \to X_2$ in terms of $\oplus, \odot, \neg$ and constants (see Theorem~\ref{th:logic} below). Moreover, we deduce from the identities in Remark~\ref{rem:multi_prod} that we could further translate any $\oplus$ operator into operations involving only  $\odot$ and $\neg$ operations.
\end{example}

More in general, one can obtain interpolators for any $m$ using the explicit indicators in the following lemma, whose proof is straightforward.

\begin{lemma}\label{lem:interp}
Given $m\in \N$ and any $j =0, \dots, m$, let
\begin{equation}\label{eq:gh}
g_j(x):=\left(x\oplus \neg(\tfrac{j}{m})\right)^m = \left(x\oplus \frac{m-j}{m}\right)^m , \quad
h_j(x):=\left(\neg(x)\oplus \frac{j}{m}\right)^m.\end{equation}
Then, 
\begin{equation*}\label{eq:gh2}
g_j(x)=\left\lbrace\begin{array}{ll} 0 & \text{if }x<\frac{j}{m}\\ 1 & \text{if }x\geq\frac{j}{m}\end{array}\right., \quad
h_j(x)=\left\lbrace\begin{array}{ll} 0 & \text{if }x>\frac{j}{m}\\ 1 & \text{if }x\leq\frac{j}{m}.\end{array}\right.\end{equation*}
 Moreover, if $0 \le j \le k \le m$, the function
 $\ell_{jk}:=g_{j} \odot h_k$ is the characteristic function of the
 subset $[\frac j m, \frac k m]=\{\frac i m, j \le i \le k\}$. That is, its value equals $1$ on the ``interval''
 $[\frac j m, \frac k m]$, and $0$ otherwise. In particular, if $j=k$, the function $\ell_{jj}$ is the indicator of the point $\frac j m$.
\end{lemma}

The following result is well known for experts in the area, but for the convenience of the reader, we give a simple proof based on the indicator functions as in the previous example.

\begin{theorem}\label{th:logic} 
Given $X_m=\left\{0,\frac{1}{m},\dots,\frac{m-1}{m},1\right\}$  and $n\in \N$.
 Every function $f:X_m^n\longrightarrow  X_m$ with variables $x_1, \dots, x_n$, is expressible in terms of  $\odot, \neg$ and 
constant functions.
\end{theorem}

\begin{proof}
Any function $f:X_m^n\to X_m$ can
be represented as
\[f(x_1,\dots,x_n)=\underset{(i_1,\dots,i_n) \in X^ n}{\bigoplus} f(i_1,\dots,i_n)\odot \ell_{i_1}(x_1)\odot \dots \odot \ell_{i_n}(x_n),\]
 where the functions $\ell_{\frac{j}{m}} =\ell_{jj}$ are defined in Lemma~\ref{lem:interp}.
It is straightforward to check that
$\ell_{i_1}(x_1)\odot \dots \odot \ell_{i_n}(x_n) = 0$ unless $(x_1, \dots, x_n)= (i_1, \dots, i_n)$, in which case it equals $1$.   

It is now enough to iteratively apply equality~\eqref{eq:oplus_as_odot} in Remark~\ref{rem:multi_prod} to get rid of the $\oplus$ operation.
\end{proof}

In fact, we can also express any function in terms of $\oplus, \neg$ and 
constant functions.
For instance, the indicators in Example~\ref{ex:f0f1} equal
\[f_0(x)=\neg(x \oplus x), \; f_1(x)=\neg(\neg(x)\oplus \neg(x)), \text{ and } f_{\frac{1}{2}}(x)=\neg(f_0(x)\oplus f_1(x)).\]
We will use in general the version
without $\oplus$ to mimic the standard choice in
the Boolean case.
On the other side, it is important to note that without multiplying by constants, Theorem~\ref{th:logic} is not true. For instance, when $m=2$, the constant function $\tfrac{1}{2}: X_2 \to X_2$ cannot be expressed in terms of  $\oplus,\odot$ and $\neg$. Indeed, by induction on the number of these operators it is easy to see that any function $f:X_2 \to X_2$ verifies that $f(0)$ equals either $0$ or $1$.

\begin{remark}\label{rem:sm}  Different functions over $\R$ can coincide on $X_m$ and there may exist more than one $\odot$-$\neg$ expression for the same function.
 \begin{enumerate}[label=(\roman*)]
\item   For instance, 
    $$ x^s =\max\{0, sx -(s-1)\} $$
    has the same  value on $X_m$ as $x^m$ for any $s \ge m$. Indeed $x^s =0$ unless $x > \frac{s-1}{s}\ge \frac {m-1}{m}$, that is, unless $x=1$ and in this case, its value is $1$ (see Figure\ref{fig:xs}).

\begin{figure}
    \centering
 \begin{minipage}{.4\textwidth}
 \centering
\begin{tikzpicture}[scale=2]
        \draw[-latex,color=darkgray,thin] (-0.2,0) -- (1.5,0) node[below] {};
        \draw[-latex,color=darkgray,thin] (0,-0.2) -- (0,1.9) node[left] {};
        \draw[dashed,color=gray] (1,1) -- (1,-0.05) node[below] {\textcolor{black}{$1$}};
        \draw[dashed,color=gray] (1,1) -- (-0.05,1) node[left] {\textcolor{black}{$1$}};
        \draw[color=gray] (0.75,0.05) -- (0.75,-0.05) node[below] {\textcolor{black}{$\frac{3}{4}$}};
        \draw[color=gray] (0.5,0.05) -- (0.5,-0.05) node[below] {\textcolor{black}{$\frac{2}{4}$}};
        \draw[color=gray] (0.25,0.05) -- (0.25,-0.05) node[below] {\textcolor{black}{$\frac{1}{4}$}};
        \draw[color=blue,very thick] (0,0) -- (0.75,0);
        \draw[color=red, thick] (0,0) -- (0.856,0);
        \draw[color=blue,very thick] plot[domain=0.75:1.15,samples=100] (\x,{4*\x-3)}) node[right] {$y=4x-3$};
        \draw[color=red, very thick] plot[domain=0.856:1.12,samples=100] (\x,{7*\x-6}) node[right] {$y=7x-6$};
 \end{tikzpicture}
\end{minipage}
\begin{minipage}{.4\textwidth}
\centering
\renewcommand{\arraystretch}{1.2}
    \begin{tabular}{|c|c|c|c|c|c|}
        \hline
        $x$ & $0$ & $\frac{1}{4}$ & $\frac{2}{4}$ & $\frac{3}{4}$ & $1$\\
        \hline
        $x^4$ & $0$ & $0$ & $0$ & $0$ & $1$\\
        \hline
        $x^7$ & $0$ & $0$ & $0$ & $0$ & $1$\\
        \hline
    \end{tabular}
\end{minipage}
    
    \caption{Comparing $x^4$ and $x^7$ when $m=4$. As real-valued functions, $\max\{0, 7x -6\}$ and $\max\{0, 4x -3\}$ are different but they coincide in $X_m= \{0,\frac{1}{4},\frac{2}{4},\frac{3}{4},1\}$.}
    \label{fig:xs}
\end{figure}

 \item   For $m=3$, the equality $x^2\odot \tfrac{1}{3}  = x\odot \frac{1}{3}$ holds for every $x\in X$. This function takes the value $1$ at $1$ and $0$ otherwise. More generally, given $m\in \N$, for every $2\leq p\leq m$ the equality $x\odot \frac{1}{m}=x^p\odot \frac{1}{m}$ holds for every $x\in X.$ 
\end{enumerate}
\end{remark}

\subsection{Motifs}\label{ssec:motifs}
We present here four simple {\em motifs}, that is, small networks that appear frequently as subnetworks of bigger ones. They are built using the multivalued logic operations for a network with $n\ge 3$  interacting biological species whose values can be described by $m+1$ values in $X_m=\{0, \frac{1}{m},\dots, \frac{m-1}{m},1\}$.

\

\noindent \textbf{\underline{Motif 1:}}
We start by describing two simple mechanisms, where node $1$ is moderately activated by node $2$. We 
show the resulting functions that describe this situation, when $m=3$.

\begin{center}
\begin{tabular}{cc}
Motif 1a: $f_1=x_1 \oplus (x_2\ominus \frac{1}{m})$ & Motif 1b:$f_1=x_1\oplus x_2^2$ \\ 
\begin{tikzpicture}[scale=2]
        \draw[-latex,color=darkgray,thin] (-0.2,0) -- (1.5,0) node[below] {$x_1$};
        \draw[-latex,color=darkgray,thin] (0,-0.2) -- (0,1.5) node[left] {$x_2$};
        \draw[color=darkgray] (0.33,0.05) -- (0.33,-0.05) node[below] {\textcolor{black}{$\frac{1}{3}$}};
        \draw[color=darkgray] (0.66,0.05) -- (0.66,-0.05) node[below] {\textcolor{black}{$\frac{2}{3}$}};
        \draw[color=darkgray] (1,0.05) -- (1,-0.05) node[below] {\textcolor{black}{$1$}};
        \draw[color=darkgray] (0.05,0.33) -- (-0.05,0.33) node[left] {\textcolor{black}{$\frac{1}{3}$}};
        \draw[color=darkgray] (0.05,0.66) -- (-0.05,0.66) node[left] {\textcolor{black}{$\frac{2}{3}$}};
        \draw[color=darkgray] (0.05,1) -- (-0.05,1) node[left] {\textcolor{black}{$1$}};
        \draw[color=darkgray] (0,0) node[below left] {$0$};
        \filldraw[gray] (0,0) circle (0.75pt) node[above right] {$0$};
        \filldraw[blue] (0.33,0) circle (0.75pt) node[above right] {$\frac{1}{3}$};
        \filldraw[magenta] (0.66,0) circle (0.75pt) node[above right] {$\frac{2}{3}$};
        \filldraw[black] (1,0) circle (0.75pt) node[above right] {$1$};
        \filldraw[gray] (0,0.33) circle (0.75pt) node[above right] {$0$};
        \filldraw[blue] (0.33,0.33) circle (0.75pt) node[above right] {$\frac{1}{3}$};
        \filldraw[magenta] (0.66,0.33) circle (0.75pt) node[above right] {$\frac{2}{3}$};
        \filldraw[black] (1,0.33) circle (0.75pt) node[above right] {$1$};
        \filldraw[blue] (0,0.66) circle (0.75pt) node[above right] {$\frac{1}{3}$};
        \filldraw[magenta] (0.33,0.66) circle (0.75pt) node[above right] {$\frac{2}{3}$};
        \filldraw[black] (0.66,0.66) circle (0.75pt) node[above right] {$1$};
        \filldraw[black] (1,0.66) circle (0.75pt) node[above right] {$1$};
        \filldraw[magenta] (0,1) circle (0.75pt) node[above right] {$\frac{2}{3}$};
        \filldraw[black] (0.33,1) circle (0.75pt) node[above right] {$1$};
        \filldraw[black] (0.66,1) circle (0.75pt) node[above right] {$1$};
        \filldraw[black] (1,1) circle (0.75pt) node[above right] {$1$};
 \end{tikzpicture} &
 \begin{tikzpicture}[scale=2]
        \draw[-latex,color=darkgray,thin] (-0.2,0) -- (1.5,0) node[below] {$x_1$};
        \draw[-latex,color=darkgray,thin] (0,-0.2) -- (0,1.5) node[left] {$x_2$};
        \draw[color=darkgray] (0.33,0.05) -- (0.33,-0.05) node[below] {\textcolor{black}{$\frac{1}{3}$}};
        \draw[color=darkgray] (0.66,0.05) -- (0.66,-0.05) node[below] {\textcolor{black}{$\frac{2}{3}$}};
        \draw[color=darkgray] (1,0.05) -- (1,-0.05) node[below] {\textcolor{black}{$1$}};
        \draw[color=darkgray] (0.05,0.33) -- (-0.05,0.33) node[left] {\textcolor{black}{$\frac{1}{3}$}};
        \draw[color=darkgray] (0.05,0.66) -- (-0.05,0.66) node[left] {\textcolor{black}{$\frac{2}{3}$}};
        \draw[color=darkgray] (0.05,1) -- (-0.05,1) node[left] {\textcolor{black}{$1$}};
        \draw[color=darkgray] (0,0) node[below left] {$0$};
        \filldraw[gray] (0,0) circle (0.75pt) node[above right] {$0$};
        \filldraw[blue] (0.33,0) circle (0.75pt) node[above right] {$\frac{1}{3}$};
        \filldraw[magenta] (0.66,0) circle (0.75pt) node[above right] {$\frac{2}{3}$};
        \filldraw[black] (1,0) circle (0.75pt) node[above right] {$1$};
        \filldraw[gray] (0,0.33) circle (0.75pt) node[above right] {$0$};
        \filldraw[blue] (0.33,0.33) circle (0.75pt) node[above right] {$\frac{1}{3}$};
        \filldraw[magenta] (0.66,0.33) circle (0.75pt) node[above right] {$\frac{2}{3}$};
        \filldraw[black] (1,0.33) circle (0.75pt) node[above right] {$1$};
        \filldraw[blue] (0,0.66) circle (0.75pt) node[above right] {$\frac{1}{3}$};
        \filldraw[magenta] (0.33,0.66) circle (0.75pt) node[above right] {$\frac{2}{3}$};
        \filldraw[black] (0.66,0.66) circle (0.75pt) node[above right] {$1$};
        \filldraw[black] (1,0.66) circle (0.75pt) node[above right] {$1$};
        \filldraw[black] (0,1) circle (0.75pt) node[above right] {$1$};
        \filldraw[black] (0.33,1) circle (0.75pt) node[above right] {$1$};
        \filldraw[black] (0.66,1) circle (0.75pt) node[above right] {$1$};
        \filldraw[black] (1,1) circle (0.75pt) node[above right] {$1$};
 \end{tikzpicture} 
\end{tabular}
\end{center}

\

\noindent \textbf{\underline{Motif 2:}}
In this motif, node $1$ depends on its previous state, node $2$ acts as an activator and node $3$ acts as a weaker activator.
\begin{align*}
  f_1 &=  x_1\oplus x_2 \oplus x_3^2\\
  &=\min\{1,x_1+x_2+\max\{0,2x_3-1\}\}.
\end{align*}

\noindent \textbf{\underline{Motif 3:}} 
In this motif, the update function of node $1$ depends on itself, an activator node $2$ and a repressor node $3$.
  \begin{align*}
  f_1 &=x_2\oplus(x_1\ominus x_3)\\
  &=x_2\oplus(x_1\odot \neg(x_3))\\
  &= \min\{1,x_2+\max\{0,x_1-x_3\}\}.
  \end{align*}

\noindent \textbf{\underline{Motif 4:}}
We finally introduce a mechanism that considers the outcome of several activators and inhibitors acting simultaneously on a single node. We propose the following motif that describes the total effect on node $\ell$ produced by activators $i\in A$ and inhibitors $j\in B$, with $A,B\subseteq \{1,\dots,n\}$, $A\cap B=\emptyset$:
$f_\ell =  \oplus_{i\in A} x_i \ominus \left(\oplus_{j\in B} x_j\right)$,
or more generally, given positive weights $w_k\in \N$:

\begin{equation*}
 f_\ell \, = \,  \underset{i\in A}{\oplus} w_i x_i \ominus \left(\underset{j\in B}{\oplus} w_j x_j\right) \, = \, \max\left\{0,\min\left\{1,\underset{i\in A}{\sum} w_ix_i \right\}-\min\left\{1,\underset{j\in B}{\sum} w_j x_j \right\}\right\}.
\end{equation*}

Note that if $\underset{j\in B}{\sum} w_j x_j \ge 1$ (i.e. when the effect of the inhibitors is big enough), then $f_\ell=0$. And if $\underset{i\in A}{\sum}w_ix_i \le 1$ and $\underset{j\in B}{\sum}w_jx_j\le 1$, then  $f_\ell>0$ if and only if $\underset{i\in A}{\sum}w_ix_i>\underset{j\in B}{\sum}w_j x_j$. This mimics Boolean threshold functions. For instance, we show three of these functions for $m=3$:
\smallskip

\begin{center}
\begin{tabular}{ccc}
$f_\ell=x_1\ominus x_2$ & $f_\ell=2x_1\ominus x_2$ & $f_\ell=x_1\ominus 2x_2$\\ 
\begin{tikzpicture}[scale=2]
        \draw[-latex,color=darkgray,thin] (-0.2,0) -- (1.5,0) node[below] {$x_1$};
        \draw[-latex,color=darkgray,thin] (0,-0.2) -- (0,1.5) node[left] {$x_2$};
        \draw[color=darkgray] (0.33,0.05) -- (0.33,-0.05) node[below] {\textcolor{black}{$\frac{1}{3}$}};
        \draw[color=darkgray] (0.66,0.05) -- (0.66,-0.05) node[below] {\textcolor{black}{$\frac{2}{3}$}};
        \draw[color=darkgray] (1,0.05) -- (1,-0.05) node[below] {\textcolor{black}{$1$}};
        \draw[color=darkgray] (0.05,0.33) -- (-0.05,0.33) node[left] {\textcolor{black}{$\frac{1}{3}$}};
        \draw[color=darkgray] (0.05,0.66) -- (-0.05,0.66) node[left] {\textcolor{black}{$\frac{2}{3}$}};
        \draw[color=darkgray] (0.05,1) -- (-0.05,1) node[left] {\textcolor{black}{$1$}};
        \draw[color=darkgray] (0,0) node[below left] {$0$};
        \filldraw[gray] (0,0) circle (0.75pt) node[above right] {$0$};
        \filldraw[blue] (0.33,0) circle (0.75pt) node[above right] {$\frac{1}{3}$};
        \filldraw[magenta] (0.66,0) circle (0.75pt) node[above right] {$\frac{2}{3}$};
        \filldraw[black] (1,0) circle (0.75pt) node[above right] {$1$};
        \filldraw[gray] (0,0.33) circle (0.75pt) node[above right] {$0$};
        \filldraw[gray] (0.33,0.33) circle (0.75pt) node[above right] {$0$};
        \filldraw[blue] (0.66,0.33) circle (0.75pt) node[above right] {$\frac{1}{3}$};
        \filldraw[magenta] (1,0.33) circle (0.75pt) node[above right] {$\frac{2}{3}$};
        \filldraw[gray] (0,0.66) circle (0.75pt) node[above right] {$0$};
        \filldraw[gray] (0.33,0.66) circle (0.75pt) node[above right] {$0$};
        \filldraw[gray] (0.66,0.66) circle (0.75pt) node[above right] {$0$};
        \filldraw[blue] (1,0.66) circle (0.75pt) node[above right] {$\frac{1}{3}$};
        \filldraw[gray] (0,1) circle (0.75pt) node[above right] {$0$};
        \filldraw[gray] (0.33,1) circle (0.75pt) node[above right] {$0$};
        \filldraw[gray] (0.66,1) circle (0.75pt) node[above right] {$0$};
        \filldraw[gray] (1,1) circle (0.75pt) node[above right] {$0$};
 \end{tikzpicture} &
 \begin{tikzpicture}[scale=2]
        \draw[-latex,color=darkgray,thin] (-0.2,0) -- (1.5,0) node[below] {$x_1$};
        \draw[-latex,color=darkgray,thin] (0,-0.2) -- (0,1.5) node[left] {$x_2$};
        \draw[color=darkgray] (0.33,0.05) -- (0.33,-0.05) node[below] {\textcolor{black}{$\frac{1}{3}$}};
        \draw[color=darkgray] (0.66,0.05) -- (0.66,-0.05) node[below] {\textcolor{black}{$\frac{2}{3}$}};
        \draw[color=darkgray] (1,0.05) -- (1,-0.05) node[below] {\textcolor{black}{$1$}};
        \draw[color=darkgray] (0.05,0.33) -- (-0.05,0.33) node[left] {\textcolor{black}{$\frac{1}{3}$}};
        \draw[color=darkgray] (0.05,0.66) -- (-0.05,0.66) node[left] {\textcolor{black}{$\frac{2}{3}$}};
        \draw[color=darkgray] (0.05,1) -- (-0.05,1) node[left] {\textcolor{black}{$1$}};
        \draw[color=darkgray] (0,0) node[below left] {$0$};
        \filldraw[gray] (0,0) circle (0.75pt) node[above right] {$0$};
        \filldraw[magenta] (0.33,0) circle (0.75pt) node[above right] {$\frac{2}{3}$};
        \filldraw[black] (0.66,0) circle (0.75pt) node[above right] {$1$};
        \filldraw[black] (1,0) circle (0.75pt) node[above right] {$1$};
        \filldraw[gray] (0,0.33) circle (0.75pt) node[above right] {$0$};
        \filldraw[blue] (0.33,0.33) circle (0.75pt) node[above right] {$\frac{1}{3}$};
        \filldraw[magenta] (0.66,0.33) circle (0.75pt) node[above right] {$\frac{2}{3}$};
        \filldraw[magenta] (1,0.33) circle (0.75pt) node[above right] {$\frac{2}{3}$};
        \filldraw[gray] (0,0.66) circle (0.75pt) node[above right] {$0$};
        \filldraw[gray] (0.33,0.66) circle (0.75pt) node[above right] {$0$};
        \filldraw[blue] (0.66,0.66) circle (0.75pt) node[above right] {$\frac{1}{3}$};
        \filldraw[blue] (1,0.66) circle (0.75pt) node[above right] {$\frac{1}{3}$};
        \filldraw[gray] (0,1) circle (0.75pt) node[above right] {$0$};
        \filldraw[gray] (0.33,1) circle (0.75pt) node[above right] {$0$};
        \filldraw[gray] (0.66,1) circle (0.75pt) node[above right] {$0$};
        \filldraw[gray] (1,1) circle (0.75pt) node[above right] {$0$};
 \end{tikzpicture} &
 \begin{tikzpicture}[scale=2]
        \draw[-latex,color=darkgray,thin] (-0.2,0) -- (1.5,0) node[below] {$x_1$};
        \draw[-latex,color=darkgray,thin] (0,-0.2) -- (0,1.5) node[left] {$x_2$};
        \draw[color=darkgray] (0.33,0.05) -- (0.33,-0.05) node[below] {\textcolor{black}{$\frac{1}{3}$}};
        \draw[color=darkgray] (0.66,0.05) -- (0.66,-0.05) node[below] {\textcolor{black}{$\frac{2}{3}$}};
        \draw[color=darkgray] (1,0.05) -- (1,-0.05) node[below] {\textcolor{black}{$1$}};
        \draw[color=darkgray] (0.05,0.33) -- (-0.05,0.33) node[left] {\textcolor{black}{$\frac{1}{3}$}};
        \draw[color=darkgray] (0.05,0.66) -- (-0.05,0.66) node[left] {\textcolor{black}{$\frac{2}{3}$}};
        \draw[color=darkgray] (0.05,1) -- (-0.05,1) node[left] {\textcolor{black}{$1$}};
        \draw[color=darkgray] (0,0) node[below left] {$0$};
        \filldraw[gray] (0,0) circle (0.75pt) node[above right] {$0$};
        \filldraw[blue] (0.33,0) circle (0.75pt) node[above right] {$\frac{1}{3}$};
        \filldraw[magenta] (0.66,0) circle (0.75pt) node[above right] {$\frac{2}{3}$};
        \filldraw[black] (1,0) circle (0.75pt) node[above right] {$1$};
        \filldraw[gray] (0,0.33) circle (0.75pt) node[above right] {$0$};
        \filldraw[gray] (0.33,0.33) circle (0.75pt) node[above right] {$0$};
        \filldraw[gray] (0.66,0.33) circle (0.75pt) node[above right] {$0$};
        \filldraw[blue] (1,0.33) circle (0.75pt) node[above right] {$\frac{1}{3}$};
        \filldraw[gray] (0,0.66) circle (0.75pt) node[above right] {$0$};
        \filldraw[gray] (0.33,0.66) circle (0.75pt) node[above right] {$0$};
        \filldraw[gray] (0.66,0.66) circle (0.75pt) node[above right] {$0$};
        \filldraw[gray] (1,0.66) circle (0.75pt) node[above right] {$0$};
        \filldraw[gray] (0,1) circle (0.75pt) node[above right] {$0$};
        \filldraw[gray] (0.33,1) circle (0.75pt) node[above right] {$0$};
        \filldraw[gray] (0.66,1) circle (0.75pt) node[above right] {$0$};
        \filldraw[gray] (1,1) circle (0.75pt) node[above right] {$0$};
 \end{tikzpicture}
\end{tabular}
\end{center}

\subsection{Fixed points} \label{ssec:fix}

When studying the dynamics of biological systems one important aspect to be considered is the long-term behavior of the system, in particular the study of the equilibria of the dynamical system. In the case of continuous-time biological systems, it is common to study the \emph{steady states} of the system. In the discrete-time context, the steady states  are called \emph{fixed points} of the system.  They are also called stable states. If moreover there is a finite number of states for every node and the dynamical system is given by $F:X_m^{n}\rightarrow X_m^{n}$,  the fixed points are the points $x\in X_m^{n}$ such that $F(x)=x$ (i.e. $f_i(x)=x_i$ for all $i\in \{1,\dots, n\}$). These points will be our main focus in the subsequent sections.

\

It is actually a sensible question to look for the fixed points of a function $F:X_m^n\to X_m^n$ with $X_m$ a finite set since most of these functions have at least one. In fact, by counting the number of these functions without fixed points the following result is immediate:

\begin{proposition}\label{prop:q}
 Set $q=(m+1)^n$, $\#X=m+1$. The proportion of $F:X_m^n\to X_m^n$ with at least one fixed point equals  
 \[
  \frac{q^q-(q-1)^q}{q^q}=1-\left(\frac{q-1}{q}\right)^q.
 \]
So, when $n\to +\infty$ or $m\to +\infty$ this proportion is close to $1-\frac{1}{e}\sim 0.6321$. 
\end{proposition}

The proof of Proposition~\ref{prop:q} is straightforward since there are $q^q$ possible functions
$F$ and there are $(q-1)^q$ functions without any fixed point.

\medskip

We describe below how we can make a function {\em smoother} in our discrete context without altering its fixed points. 

Following~\cite{chifman} and its references, we  introduce a function to {\em preserve continuity}, which means in our setting that each node will change at most $\frac{1}{m}$ in one time step. These {\em unitary networks} were introduced in~\cite{Schaub} with the idea that they better capture the continuity of biological interactions. For this purpose, we take into account the previous state of the corresponding node (we add a self-regulation loop to each network node). The future value of the regulated variable under continuity is computed as follows: 

\begin{definition} \label{def:cont}
Given $X_m=\{0,\frac{1}{m},\dots,1\}$ we define the function $h:X_m^2\to X_m$ as
\begin{equation}\label{eq:h_cont}
 h(x,y)=\begin{cases}
          x\oplus \frac{1}{m} & \text{if } y>x\\
          x \qquad & \text{if } y=x\\
          x\ominus \frac{1}{m} & \text{if } y<x.
         \end{cases}
\end{equation}
Given $F = (f_1, \dots, f_n): X_m^n \to X_m^n$, the continuous version of each coordinate function $f_i: X_m^n \to X_m$  is denoted by $f_{i,cont}$ and is defined as 
\begin{equation}\label{eq:fcont}
  f_{i,cont}(x)=h(x_i,f_i(x)), 
\end{equation}
for $i=1,\dots, n$.
We denote $F_{cont}=(f_{1, cont},\dots, f_{n, cont})$.
\end{definition}

The following lemma is an immediate consequence of the previous definition:

\begin{lemma}\label{lem:continuity}
A function $F: X_m^n \to X_m^n$ and its associated function $F_{cont}$ in Definition~\ref{def:cont}
have the same fixed points. 
\end{lemma}

For any $m$, the continuous version of  all the powers $x^k$ is the same function, independently of $k$.
\begin{lemma}
 Given $k\in \N_{\geq 2}$, then 
 \begin{equation}x^k_{cont}=\begin{cases}
          x\ominus \frac{1}{m} & {\rm if \, }\, x<1\\
          1  & {\rm if \,}\, x=1.\\
         \end{cases}\end{equation}
\end{lemma}

\begin{proof}
It is straightforward to  check that         
$x\geq x^k$  for any $x \in X$ and that the equality holds if only if $x=0$ or $x=1$.
\end{proof}

\section{From any vector function to a
\texorpdfstring {$\odot$-$\neg$}{dot-neg} network function}\label{sec:dotneg}

From now on, we will focus on dynamical systems over the set $X_m~=~\{0,\tfrac{1}{m}, \tfrac{2}{m}, \dots, \tfrac{m-1}{m}, 1\}$. Inspired by the framework in~\cite{VC}, we  propose a novel method to compute the fixed points of a network. For this, we now introduce the notion of $\odot$-$\neg$ functions. Algorithm~1 can be used to translate an input network into one given by $\odot$-$\neg$ functions,  and we show in Theorem~\ref{thm:DN} how they can be used to compute the steady states of the original network.

\smallskip

We  mimic a procedure which has been proved useful in Boolean networks. Recall that every Boolean function $f$ can be written in conjunctive normal form as 
 $$f=w_1\wedge \dots \wedge w_r$$ where 
 $$w_j=y_{i_1}\vee \dots \vee y_{i_{n_j}}$$ 
 for $j\in \{1,\dots,r\}$,  $\{i_1,\dots, i_{n_{j}}\}\subseteq \{1,\dots, n\}$ and $y_{k} \in \{ x_k, \neg(x_k)\}$.
This gives the idea of AND-NOT functions: a vector function $F=(f_1,\dots,f_n):X_m^n \to X_m^n$ such that every $f_i:X_m \to X_m$ is an AND-NOT function, has the same information as its wiring diagram~\cite{VC, VCLA}. 
In order to obtain similar properties in our context, we define what we call $\odot$-$\neg$ functions.

\begin{definition}\label{def:dotneg}
  A function $f:X_m^n\longrightarrow X_m$ is called  a $\odot$-$\neg$ function if it can be written as 
  $$f(x_1,\dots,x_n)=\bigodot_{j\in A}x_j^{p_{j}} \odot\bigodot_{j\in B}\neg(x_j)^{q_{j}}\bigodot \frac c m$$
  with $c\in X_m$, $A\cup B\subseteq \{1,\dots,n\}$ and $\{p_{j},q_{j}\}\subseteq \N$ .

  A vector function $F=(f_1,\dots,f_n):X_m^n\longrightarrow X_m^n$ is called an $\odot$-$\neg$ network function if $f_i$ is a  $\odot$-$\neg$ function for all $i \in \{1,\dots, n\}$.
\end{definition}

The most important feature of $\odot$-$\neg$ networks is that there is a direct correspondence between their wiring diagram and the network functions. All the information about the network dynamics can be read from the wiring diagram (and vice versa), which is defined by a labeled graph $G=(V_G,E_G)$ with vertices 
$$V_G=\left\{1,\dots,n,C\right\}$$ 
($\{1,\dots,n\}$ may also be denoted by $\{x_1,\dots,x_n\}$) and edges $E_G$ given as follows. If $F=(f_1,\dots,f_n)$ and 
$$f_k(x_1,\dots,x_n)=\bigodot_{i\in A_k}x_i^{p_{k_i}} \odot\bigodot_{j\in B_k}\neg(x_j)^{q_{k_j}}\bigodot \frac {c_k} m,$$ then there exist labeled edges of the form 

\begin{equation*}
\xymatrix{x_i\ar@{->}[r]^{p_{k_i}}&x_k} \, \text{if} \, p_{k_i}> 0;\quad
\xymatrix{x_j\ar@{-|}[r]^{q_{k_j}}& x_k} \, \text{if} \, q_{k_j}> 0; \quad
\xymatrix{C\ar@{->}[r]^{\frac{c_k} m}&x_k }.
\end{equation*}

Not every function is a  $\odot$-$\neg$ function. Consider for instance $f_1=\neg(x_i\odot x_j)$ for any $i, j$.

\begin{example} \label{ex:w} 
Given 
 $F(x_1,x_2,x_3)= (x_2^2\odot \neg (x_3)^q, x_1,\tfrac{1}{2}\odot x_1^p)$, the  associated wiring diagram is as follows:
\[
\xymatrix{
&C\ar[ld]_{\frac{1}{2}}
\ar@{-->}[d]^1\ar@{-->}[rd]^1
& \\ 
x_3\ar@{-|}[r]^q& x_1\ar@{->}@/_0.7pc/[r]_1\ar@{->}@/^0.7pc/[l]^p&x_2\ar@{->}[l]_2}
\]
The  dashed edge can be skipped as $1\odot x=x$ for every $x$. Also, in a general wiring diagram, if there is an edge $C \overset{0}{\to} x_i$, we can delete all the other edges with end point $x_i$
because $f_i$ is then identically zero.
\end{example}

\begin{proposition} \label{prop:c}
Given  a $\odot$-$\neg$ function, $f\not\equiv 0$, written as in Definition \ref{def:dotneg}, the sets $A,B$ can be chosen such that $A\cap B=\emptyset$, $c\neq 0$ and $1\leq p_j, q_j\leq c$.
Moreover, this bound in the exponents is sharp. 
\end{proposition}

\begin{proof}
First note that  if $A\cap B\neq\emptyset$ then $f=0$ since $x \odot \neg (x) =0$ for any $x$ by ~\eqref{eq:minus}. 
 If $A \neq \emptyset$, consider $\ell\in A$ and $x\in X_m^n$ with $x_i = 1$ for all $i\in A$, $i\neq \ell$, and $x_j = 0$ for all $j\in B$. Then, $f(x) = x_\ell^{p_{\ell,k}} \odot \frac{c}{m}$.
We then need to see that if $p \geq c$, then
\begin{equation}\label{eq:xell}
  x_\ell^p \odot \frac{c}{m} = x_\ell^c \odot \frac{c}{m}.
\end{equation}
    If $x_\ell = 1$ or $x_\ell=0$ it is immediate to see that the equality holds. If $x_\ell \in X_m \backslash \{0, 1\}$ we next prove that both sides of the equation equal zero. Since $x_\ell^p \odot \frac{c}{m} = \max\{0,px+\frac{c}{m}-p\}$, we have:
    \[
        x_\ell^p \odot \frac{c}{m} >0 \Leftrightarrow p < \frac{c}{m} \frac{1}{(1 - x_\ell)}.
    \]
    As $x_\ell\in X_m \backslash \{0, 1\}$ then  $x_\ell = \frac{c'}{m}$ with $1\leq c' \leq m-1$. Then
    \[
        \frac{c}{m} \frac{1}{(1 - x_\ell)}=\frac{c}{m} \frac{1}{(1 - \frac{c'}{m})}=\frac{c}{m-c'}\leq c \leq  p,
    \]
    and this means $x_\ell^p \odot \frac{c}{m} = x_\ell^c \odot \frac{c}{m}=0$.  Analogously,~\eqref{eq:xell} is also true if we replace $x_\ell$ by $\neg(x_\ell)$.

    To see that the bound $p_{j}, q_j \le c$ is sharp, it is enough to note that $(\tfrac{m - 1}{m})^{c-1} \odot \frac c m = \frac 1 m \neq 0 = (\tfrac{m - 1}{m})^{c} \odot \frac c m $.
\end{proof}

\medskip

In fact, the following uniqueness result holds, giving a canonical form to write a $\odot$-$\neg$ function.

\begin{theorem}\label{th:unique}
Given a $\odot$-$\neg$ function $f: X_m^n \to X_m$
  \begin{equation} \label{eq:fx}
  f(x)=\bigodot_{j\in A}x_j^{p_{j}} \odot\bigodot_{j\in B}\neg(x_j)^{q_{j}}\bigodot \frac c m,
  \end{equation}
  with $c\in \{0, \dots, m\}$, $A\cup B\subseteq \{1,\dots,n\}$ and $\{p_{j},q_{j}\}\subseteq \N$,
  then either $f \equiv 0$ and we write $f= \frac 0 m$, or we can choose $A$ and $B$ disjoint, and we have that
    \begin{equation} \label{eq:min}
    f(x)=\bigodot_{j\in A}x_j^{\min\{p_{j},c\}} \odot\bigodot_{j\in B}\neg(x_j)^{\min\{q_{j},c\}}\bigodot \frac c m.
    \end{equation}
Moreover, this expression is unique.
\end{theorem}

\begin{proof}
Given a nonzero function $f$ as in~\eqref{eq:fx}, the existence of subsets $A$, $B$ and exponents yielding equality~\eqref{eq:min} is the content of Proposition~\ref{prop:c}.

We now prove that an expression like~\eqref{eq:min} is unique. Given nonzero $\odot$-$\neg$ functions $f_1, f_2$ written as
\[
f_k=\underset{i\in A_k}{\bigodot} \, x_i^{p_{i,k}}\,  \, \underset{j \in B_k}{\bigodot} \, \neg(x_j)^{q_{j,k}}\,  \, \bigodot \frac{c_k}{m}, \quad k\in\{1,2\},
\]
with $A_k,B_k\subseteq\{1,\dots,n\},  \; A_k\cap B_k=\emptyset$ and $1\leq p_{i,k} \leq c_k$, $1\leq q_{j,k} \leq c_k$ for all $i\in A_k$, $j\in B_k$,
we need to see that the following conditions are equivalent:
\begin{itemize}
    \item[(i)] $f_1(x) = f_2(x)$ for all $x \in X_m^n$,
    \item[(ii)]  $c_1=c_2=c \neq 0$, 
$A_1=A_2$, $B_1=B_2$, $p_{i_1}=p_{i,2}$ for all $i\in A_1$, and $q_{j,1}=q_{j,2}$ for all $j\in B_1$.
\end{itemize}
It is clear that (ii) implies the equality of the functions in (i).

Assume (i) holds with $c_1, c_2 \neq 0$.
If  the four subsets $A_1, A_2, B_1, B_2$ are all empty, then $c_1=c_2$ and (ii) holds.
If there exists $\ell \in A_1 \cap B_2$, consider $x\in X_m^n$ such that $x_i = 1$ for all $i\in A_1$, and $x_j = 0$ for all $j\in B_1$; then $f_1(x)=\frac{c_1}{m}\neq 0$ and $f_2(x)=0$ which is not possible. Therefore,  $A_1\cap B_2=\emptyset$, and in a similar way $A_2\cap B_1=\emptyset$.
 
If there exists $\ell\in A_1\setminus A_2$ consider $x\in X_m^n$ such that $x_i = 1$ for all $i\in A_2$, and $x_j = 0$ for all $j\notin A_2$, then $f_1(x)= 0$ but $f_2(x)=\frac{c_2}{m}\neq 0$ which is not possible. By a similar reasoning, we get that $A_1=A_2$ and $B_1=B_2$. 
Choose $x\in X_m^n$ with  $x_i = 1$ for all $i\in A_1=A_2$ and $x_j = 0$ for all $j \in B_1=B_2$. Then  $f_k(x) = \frac{c_k}{m}$, $k=1,2$, and we deduce from (i) that $c_1=c_2=c \in \{0,\dots,m\}$.
To see that the exponents are equal, if $A_1\neq \emptyset$, consider $\ell\in A_1$ and $x\in X_m^n$ such that $x_i = 1$ for all $i\in A_1$, $i\neq \ell$, and $x_j = 0$ for all $j\in B_1$; then $f_k(x) = x_\ell^{p_{\ell,k}} \odot \frac{c}{m}$, $k=1,2$. 
Setting $x_\ell = \frac {m-1} m$, we have
 \[
    f_k(x) = \left(\frac{m-1}{m}\right)^{p_{\ell,k}}\odot \frac{c}{m}= \max\left\lbrace 0, \frac{c}{m} - p_{\ell,k}\left(1-\frac{m-1}{m}\right)\right\rbrace = \max\left\lbrace 0, \frac{c - p_{\ell,k}}{m} \right\rbrace = \frac{c - p_{\ell,k}}{m},
 \]
since $c \geq p_{\ell,k}$. Then $p_{\ell, 1}=p_{\ell,2}$. In a similar way if $B_1 \neq \emptyset$, necessarily $q_{j,1}=q_{j,2}$ for all $j\in B_1$.
\end{proof}

Before we show that any network function $F:X_m^n\to X_m^n$ can be transformed into a $\odot$-$\neg$ network function  where we can trace the fixed points of $F$, we introduce a complexity measure that will estimate the size of the new $\odot$-$\neg$ network.

\begin{definition}\label{def:mu} 
 Given a function $f:X_m^n \to X_m$ written in terms of $\odot$ and $\neg$, we define by Algorithm~\ref{algo:mu_and_barf} below a complexity measure   called \emph{depth}, denoted by $\mu(f)$. We also define the associated function $\overline{f}$ depending on $\mu(f)$ new variables.
 
 \begin{algorithm}[htb]\label{algo1}
 \caption{How to build $\mu$ and $\overline{f}$}\label{algo:mu_and_barf}
  \flushleft{\textbf{Input:} $f:X_m^n\to X_m$,  written in terms of $\odot$ and $\neg$}.
  \\
  \textbf{Output:} $\mu(f)$ and a network function $(\overline{f},\overline{f}_{u_1},\dots,\overline{f}_{u_{\mu(f)}}):X_m^{n+\mu(f)}\to X_m^{1+\mu(f)}$.
  \begin{algorithmic}
    \State $\mu(f) \gets 0$
    \State $\overline{f}_0 \gets f$
    \State $k \gets 1$
    \While{there is an expression $\neg(h)$, with $h$ a $\odot$-$\neg$ function with two or more factors, in $\overline{f}_{k-1}$}
    \State Add a new variable $u_k$
    \State $\overline{f}_{k} \gets$ the resulting function after replacing $h$ in $\overline{f}_{k-1}$ with the new variable $u_k$ 
    \State $\overline{f}_{u_k} \gets h$
    \State $\mu(f) \gets \mu(f)+1$
    \State $k \gets k+1$
    \EndWhile
    \State $\overline{f} \gets \overline{f}_{k}$
  \end{algorithmic}
\end{algorithm}

Given a network function $F=(f_1,\dots, f_n):X_m^{n}\to X_m^{n}$, apply the algorithm to every $f_i$, and we define $\mu(F)=\sum_{i=1}^{n} \mu(f_i)$. 
This is an upper bound to the total amount of new variables one needs to add in order to write each $F_i$ as a $\odot$-$\neg$ function as in Definition~\ref{def:dotneg}.
\end{definition}

\begin{remark}\label{rem optimized}
 Algorithm~\ref{algo:mu_and_barf} can be optimized, for example, by making a list of substitutions.  A new variable is added only if the new  substitution is not on the list. 
\end{remark}

 \begin{remark}\label{rmk:recursive}
     For $f:X_m^n\to X_m$ and $1\leq k\leq \mu(f)$, note that $\overline{f}_{u_k}$ only depends on previously defined variables: $\overline{f}_{u_1}=\overline{f}_{u_1}(x_1,\dots,x_n)$, and  $\overline{f}_{u_k}=\overline{f}_{u_k}(x_1,\dots, x_n,u_1,\dots,u_{k-1})$, for $2\leq k$.
 \end{remark}

\begin{example}\label{ex:f_mu}
If we go back to Motifs~2,~3 and~4 from \S~\ref{ssec:motifs}, we can transform each function into a $\odot$-$\neg$ function, by reproducing the steps in Algorithm~\ref{algo:mu_and_barf}.

\

\noindent \textbf{\underline{Motif 2:}}

\noindent \begin{tabular}{r@{\hspace{1mm}}c@{\hspace{1mm}}lcc@{\hspace{2mm}}c@{\hspace{2mm}}l}
    $f_1$ & $=$ & $x_1\oplus x_2 \oplus x_3^2$ & & $\overline{f}_1$ & $=$ & $\neg(u_2)$ \\
    & $=$ & $\min\left\{1,x_1+x_2+\max\{0,2x_3-1\}\right\}$ & $\Rightarrow$ & $\overline{f}_{u_1}$ & $=$ &  $x_3\odot x_3$\\
    &$=$ & $\neg\left(\neg(x_1)\odot \neg(x_2)\odot \neg(x_3\odot x_3)\right)$ & & $\overline{f}_{u_2}$ & $=$ & $\neg(x_1)\odot \neg(x_2)\odot \neg(u_1)$
\end{tabular}

\

\noindent \textbf{\underline{Motif 3:}} 

\noindent \begin{tabular}{r@{\hspace{1mm}}c@{\hspace{1mm}}lcc@{\hspace{2mm}}c@{\hspace{2mm}}l}
    $f_1$ & $=$ & $x_2\oplus(x_1\ominus x_3)$ & & $\overline{f}_1$ & $=$ & $\neg(u_2)$ \\
    & $=$ & $\min\left\{1,x_2+\max\{0,x_1-x_3\}\right\}$ & $\Rightarrow$ & $\overline{f}_{u_1}$ & $=$ &  $x_1\odot \neg(x_3)$\\
    &$=$ & $\neg\left(\neg(x_2)\odot \neg(x_1\odot \neg(x_3))\right)$ & & $\overline{f}_{u_2}$ & $=$ & $\neg(x_2)\odot \neg(u_1)$
\end{tabular}

\

\noindent \textbf{\underline{Motif 4:}} 

\noindent \begin{tabular}{r@{\hspace{1mm}}c@{\hspace{1mm}}lcc@{\hspace{2mm}}c@{\hspace{2mm}}l}
    $f_\ell$ & $=$ & $\bigoplus_{i\in A} w_i x_i \ominus \left(\bigoplus_{j\in B} w_j x_j\right)$ & & $\overline{f}_\ell$ & $=$ & $\neg(u_1)\odot \bigodot_{j\in B} \neg(x_j)^{w_j}$ \\
    & $=$ & $\max\left\{0,\min\left\{1,\underset{i\in A}{\sum} w_ix_i \right\}-\min\left\{1,\underset{j\in B}{\sum} w_j x_j \right\}\right\}$ & $\Rightarrow$ & $\overline{f}_{u_1}$ & $=$ &  $\bigodot_{i\in A}  \neg(x_i)^{w_i}$\\
    &$=$ & $\neg\left(\bigodot_{i\in A}  \neg(x_i)^{w_i}\right) \odot \left(\bigodot_{j\in B} \neg(x_j)^{w_j}\right)$ & &  & & 
\end{tabular}

\medskip

In the three cases, the corresponding value of $\mu$ equals $2$ and we can see that this is related to the number of alternations between maximum and minimum occurring in the functions. 
\end{example}

\begin{remark} \label{rem:branches}
Note that by items~\ref{it:several_products} and~\ref{it:several_sums}  in Proposition~\ref{prop:oper}, consecutive  $\oplus$ (resp. $\odot$) operations can be replaced by a single minimum (resp. maximum). Thus, given any function $f:X_m^n \to X_m$, $\mu(f)$ {\em measures} the alternation of $\oplus$ and $\odot$ operations in the expression of $f$.

 We could have written any function in terms of $\oplus, \neg$ and constant functions and proposed a similar algorithm to add a new variable for every negation of a sum. We chose $\odot$-$\neg$ functions instead of $\oplus$-$\neg$ ones to generalize \cite{VC}. 
\end{remark}

Wiring diagrams of $\odot$-$\neg$ networks encode all the information that the network carries.
The fixed points of the dynamical system obtained by iteration of $F$ are 
in bijection with the fixed points of $ \overline{F}$.  Given $F: X_m^{n_1}\to X_m^{n_1}$ and the corresponding $\overline{F}:X_m^{n_2}\to X_m^{n_2}$ constructed as in Definition~\ref{def:mu}, call $x=(x_1,\dots,x_{n_1})$, we define the function $U: X_m^{n_1}\to X_m^{n_2-n_1}$ in a recursive way (see Remark~\ref{rmk:recursive}):
\begin{equation}\label{def:U}
 \begin{cases}
     U_1(x)= \overline{f}_{u_1}(x),\\
     U_k(x)=\overline{f}_{u_k}(x,U_1(x),\dots,U_{k-1}(x)), \text{ for } 2\leq k\leq n_2-n_1.
 \end{cases}
\end{equation}
\

\begin{theorem}\label{thm:DN} 
 Given $F:X_m^{n_1}\to X_m^{n_1}$, the algorithm in Definition~\ref{def:mu} constructs $\overline{F}:X_m^{n_2}\to X_m^{n_2}$ which consists of $\odot$-$\neg$ functions and the function  $U:X_m^{n_1}\to X_m^{n_2-n_1}$ as in~\eqref{def:U}, where $n_2-n_1\leq \mu(F)$.  Moreover, the following diagram commutes:
 \begin{equation}\label{eq:diagram_DN}
\xymatrix{
 X_m^{n_1}\ar@{->}[r]^{(Id, U)}&X_m^{n_2}\ar@{->}[d]^{\overline{F}}\\
 X_m^{n_1}\ar@{<-}[u]^{F}&X_m^{n_2}\ar[l]^{\pi(x,y)=x}\\
}
 \end{equation}
 and the fixed points of $F$ are in bijection with the fixed points of $\overline{F}$ via projection onto the first $n_1$ coordinates.
\end{theorem}

Before the proof we present a clarifying example.

\begin{example}\label{ex:TD}
This example is based on Thomas and D'Ari \cite[Chapter~4,\S II]{Thomas_DAri}. Consider three genes $\mathcal X,\mathcal Y, \mathcal Z$ with respective products denoted by $x, y, z$. Gene $\mathcal X$ is expressed constitutively (not regulated), gene $\mathcal Y$ is expressed only in the absence of product $x$, and gene $\mathcal Z$ is expressed provided product $y$ or product $z$ is present.

The Boolean logical relations proposed in \cite{Thomas_DAri} are 
\[
(x,y,z) \mapsto (1, \lnot x, y \lor z),
\]
with two steady states: $(1,0,0)$ and $(1,0,1)$. The authors argue that starting from the initial state $(0,0,0)$ if  gene $ \mathcal X$ is turned on (e.g. immediately after infection by a virus) then gene $\mathcal Y$ will be expressed until product $x$ appears, switching it off. Gene $\mathcal Z$ will be turned on only if $y$ appears before $x$ switches gene $Y$ off. 
Note that the order in which individual variables are updated does not change the steady states of the system.

This model is assumed to have three time delays $t_x, t_y, t_z$. If $t_x<t_y$, then $(0,0,0) \mapsto (1,0,0)$ and the system remains there as in the synchronous Boolean case. 

\smallskip

When $t_x>t_y$ then $(0,0,0)\mapsto (0,1,0)$. If moreover $t_x>t_y+t_z$, then $(0,1,0)\mapsto (1,1,1)\mapsto (1,0,1)$ and the system remains at $(1,0,1)$. We propose instead the following {\em synchronous} modeling with $m=3$.  We assume that $\mathcal X$ switches $\mathcal Y$ off completely only if $x$ is at level $1$ (fully activated). We then consider the following network function $F =(f_x, f_y, f_z)$:
\[
F: X_m^3 \to X_m^3, \quad  F(x,y,z) \, = \, \left(x \oplus \tfrac{1}{3}, \neg(x), y \oplus z\right).
\]
The orbit of $(0,0,0)$ under $F$ equals: $(0,0,0) \mapsto (\tfrac{1}{3},1,0) \mapsto (\tfrac{2}{3}, \tfrac{2}{3}, 1) \mapsto (1, \tfrac{1}{3}, 1) \mapsto (1, 0,1) \mapsto (1,0,1)$. So, $(1,0,1)$ is a fixed point and as $\mathcal X$ is activated in three steps, $\mathcal Z$ can be turned on even if $\mathcal Y$ is eventually turned off.   Note that $(1,0,z_0)$ is a fixed point for any value $z_0$ of $z$.
 
To translate $F$ to an $\odot$-$\neg$ function $\overline{F}$  we need to add two variables $u_1$, $u_2$. Thus,  we have that $\overline{F}= (\overline{f}_x, \overline{f}_y, \overline{f}_z, \overline{f}_{u_1},\overline{f}_{u_2}):X_m^5 \to X_m^5$ is defined by

\begin{center}
   \begin{tabular}{c@{\hspace{2mm}}c@{\hspace{2mm}}l}
    $\overline{f}_x$ & $=$ & $\neg(u_1)$ \\
    $\overline{f}_y$ & $=$ & $\neg(x)$ \\
    $\overline{f}_z$ & $=$ & $\neg(u_2)$ \\
    $\overline{f}_{u_1}$ & $=$ &  $\neg(x)\odot \tfrac{2}{3}$\\
    $\overline{f}_{u_2}$ & $=$ & $\neg(y)\odot \neg(z)$.
  \end{tabular} 
\end{center}

Its fixed points are in bijection with those of $F$ by projecting into the first $3$ coordinates. Indeed, if $(x^*,y^*,z^*,u_1^*,u_2^*)$ is a fixed point of $\overline{F}$ then
\begin{align*}
    x^* &=\neg(u_1^*)=\neg\left(\neg(x^*)\odot \tfrac{2}{3}\right)=x^*\oplus \tfrac{1}{3}= f_x(x^*,y^*,z^*),\\
    y^* &=\neg(x^*)= f_y(x^*,y^*,z^*),\\
    z^* &=\neg(u_2^*)=\neg\left(\neg(y^*)\odot \neg(z^*)\right)=y^*\oplus z^*= f_z(x^*,y^*,z^*),
\end{align*}
and $(x^*,y^*,z^*)$ is a fixed point of $F$. Reciprocally, if $(x^*,y^*,z^*)$ is a fixed point of $F$, define $u_1^*=\neg(x^*)\odot \tfrac{2}{3}$ and $u_2^*=\neg(y^*)\odot \neg(z^*)$ and it is straightforward to check that $(x^*,y^*,z^*,u_1^*,u_2^*)$ is a fixed point of $\overline{F}$.
\end{example}

\begin{proof}[Proof of Theorem~\ref{thm:DN}]
From Algorithm~\ref{algo:mu_and_barf} and Definition \ref{def:mu} 
it is straightforward to check that $n_2-n_1\leq \mu(F)$. It is also immediate from the definition of $\overline{F}$ and $U$ that $\pi(\overline{F}(x,U(x)))=F(x)$ for any $x\in X^{n_1}$ and the diagram in~\eqref{eq:diagram_DN} commutes. 

Consider a fixed point $x^*$ of $F$ and define $u^*=U(x^*)$. Then, $\overline{f}_i(x^*,u^*)=\overline{f}_i(x^*,U(x^*))=f_i(x^*)=x_i^*$ for  $1\leq i\leq n_1$, and $\overline{f}_i(x^*,u^*)=u_i^*$ for  $1\leq i\leq n_2-n_1$ by definition of $u^*$. Thus, $(x^*,u^*)$ is a fixed point of $\overline{F}$.

Conversely, if $(x^*,u^*)$ is a fixed point of $\overline{F}$, by the recursive definition~\eqref{def:U} of $U$ we have that $u^*=U(x^*)$. Then, $F(x^*)=\pi\left(\overline{F}\left(x^*,U(x^*)\right)\right)=\pi\left(\overline{F}(x^*,u^*)\right)=\pi(x^*,u^*)=x^*$. We conclude that $x^*$ is a fixed point of $F$. 
\end{proof}

\section{Simplifying the fixed point equations}\label{sec:reductions}

We propose in this section several reductions that can be applied to a multivalued network $F:X_m^{n}\to X_m^{n}$, and its $\odot$-$\neg$ corresponding network $\overline{F}$, that eventually produce a $\odot$-$\neg$ network $G:X_m^{n'}\to X_m^{n'}$ whose fixed points allow us to reconstruct immediately the fixed points of $F$. Usually $n'\ll n$, see for instance \S~\ref{ssec:mammalian} ($n=19$, $n'=4$) or \S~\ref{sec:bioexample} ($n=15$, $n'=5$ in the worst scenario).

\smallskip

We first simplify the equations with the aid of some reductions that can be easily read from the algebraic expressions. We then transform the reduced network into a $\odot$-$\neg$ network, in which new variables are added to gain simplicity when dealing with minima and maxima imposed by the logical frameworks. We can then apply more network reductions to the $\odot$-$\neg$ network before dealing with the (usually small number of) fixed point equations.

 \subsection{Network reductions}\label{ssec:reductions}

Several network reductions  are proposed in~\cite{Ve2,VCLA} for Boolean networks. In our multivalued context some  of these reductions are still valid. Here we provide a list of simple reductions with the objective of reducing the size of the network by elimination of variables in the process of finding the fixed points.  After finding the fixed points of the smaller network, the fixed points of the original network are reconstructed using the information of the variables that have been eliminated.

 \medskip

 The reductions presented below generalize the ones in \cite{VCLA} for AND-NOT Boolean networks. In what follows, $w, w', w''$ will denote (any/suitable)  multivalued functions.

\begin{enumerate}[label=(\roman*)]
 \item \label{it:xi_negxi} 
 If $f_i=c$ with $c \in X_m$, $f_i=x_k$, or $f_i=\neg(x_k)$ with $i \neq k$, then $x_i$ 
 can be eliminated by replacing $x_i=c$, $x_i=x_k$ or $x_i=\neg(x_k)$, respectively, in every other function.
 \item If  $f_i=f_j$, then we can set $f_i=x_j$ and eliminate $x_i$. 
 \item If $x_i$ does not appear in any $f_j$, then $x_i$ can be eliminated.
 \item As for all $x$ we have the identity $x\odot \neg(x)=0$ by~\eqref{eq:minus}, if
 $f_i=x_j\odot x_k\odot w$, and $f_j=\neg(x_k)\odot w'$, then at a fixed point we have that $f_i=x_k\odot \neg(x_k)\odot w''=0$. In this case, we set $x_i=0$ and eliminate $x_i$.
 \item By Remark~\ref{rem:sm}, we can replace any power $s >m$ by $m$.
 \end{enumerate}

 Recall that there is a direct correspondence between the wiring diagram of a $\odot$-$\neg$ network  and the network functions. We present in the following proposition a new reduction that can be easily detected from the wiring diagram of such a network. 
 
 \begin{proposition}\label{prop:pos_neg_reduction}
  
Let $F:X_m^n \to X_m^n$ be a $\odot$-$\neg$ network. If the associated wiring diagram contains (at least) two directed paths ending at $x_k$ and starting at  the same node $x_j$, such that both paths have only positive arrows  except for only one negative arrow leaving from $x_j$, then $x_k =0$ for any steady state $x$.  
 
 \[
\xymatrix{
&&\dots&\ar@{->}[rd]^q&\\
x_j\ar@{->}[ru]^s\ar@{-|}[rd]^r
&&&& x_k\\
&&\dots&\ar@{->}[ru]^t&
}
\hspace{8mm}
\xymatrix{\\
\Rightarrow
\\}
\hspace{8mm}
\xymatrix{&\\
C\ar[r]^0& x_k
}
\]
  
 \end{proposition}
The proof of Proposition~\ref{prop:pos_neg_reduction} is straightforward  since at any steady state $x$ we can eventually replace  $f_k=x_j\odot \neg(x_j)\odot w=0$.
For example, if $f_4(x)= x_1 \odot x_2 \odot x_4$, $f_3(x)=x_1$, $f_2(x)= \neg(x_1)^2\odot \neg(x_3)$,  and
$f_1(x)= x_4^2$,  then we have two directed paths from $x_1$ to $x_4$ and at a steady state $x^*$ we 
can make the following replacements: $x^*_4 = x^*_1 \odot x^*_2 \odot x^*_4 =  x^*_1 \odot \neg(x_1^*)^2\odot
\neg(x_3^*) \odot x^*_4 = 0$.
Useful examples of the reductions above can be seen in the following sections.

 \begin{remark}\label{rem:odot}
Note that all the proposed reductions 
preserve the property of having a $\odot$-$\neg$ function.
\end{remark}

 \begin{definition}\label{indeg}
We define the indegree $\indeg(F)$ of an $\odot$-$\neg$ network function $F: X_m^n \to X_m$ as follows. $F$ is represented by a digraph with nodes  in the set $\{x_1,\dots,x_{n}, C\}$, and directed edges
of the form
$\xymatrix{C \ar@{->}[r]^{c} & x_j}$, $\xymatrix{x_i \ar@{->}[r]^{p} & x_j}, \xymatrix{x_i \ar@{-|}[r]^{q} & x_j}$, for some $p, q\in \N$. We define $\indeg(f_i)=\#\{\textit{arrows arriving to}\  x_i\}$
and $\indeg(F)=\max\{\indeg (f_i)\}$. 
\end{definition} 
Thus, the indegree measures the maximum number of variables on which any coordinate function
$f_i$ of $F$ depend.  For instance, if $F$ is the function in Example~\ref{ex:w}, we have that $\indeg(f_1)= \indeg(f_3)=2$, $\indeg(f_2)=1$ and $\indeg(F)=2$.

 In the Boolean case, most of the reductions described in \cite{VCLA} do not increase the indegree.
Applying any sequence of the reductions we introduced before in this section, we could reduce the number of variables and the computation of the fixed points of a $\odot$-$\neg$ network
function $F$ to the computation of the fixed points of a reduced $\odot$-$\neg$ function without increasing the indegree:
 
 \begin{proposition}\label{prop:indegree}
The indegree of any network function obtained by any sequence of the described reductions from a $\odot$-$\neg$ network function $F$ has indegree smaller of equal than ${\rm indeg}(F)$.
\end{proposition} 
The proof of Proposition~\ref{prop:indegree} is straightforward.

\subsection{Example from Cell Collective: Mammalian Cell Cycle (1607)}\label{ssec:mammalian}

Here we generalize a  Boolean network described in $\hbox{https://cellcollective.org//\#module/1607:1/mammalian-cell-cycle/1}$, where there are many references to the use of this network.
It models the transmembrane tyrosine kinase ERBB2 interaction network. The overexpression of this protein is an adverse prognostic marker in breast cancer, and occurs in almost 30\% of the patients. 

Using our reductions, we can reduce from 19 original variables to 4 variables and a constant.  Yet, in one of the cases, the computation is too long to be done by hand and we used our implementation~\cite{Aye}, that gives an immediate answer in this case

\smallskip

We call $c=\text{EGF}$, $x_1=\text{ErbB1}$, $x_2=\text{ERa}$, $x_3=\text{CycE1}$, $x_4=\text{CycD1}$, $x_5=\text{CDK4}$, $x_6=\text{ErbB3}$, $x_7=\text{cMYC}$, $x_8=\text{Akt1}$, $x_9=\text{ErbB2}$, $x_{10}=\text{p21}$, 
$x_{11}=\text{ErbB1\_2}$, $x_{12}=\text{ErbB1\_3}$, $x_{13}=\text{IGF1R}$, $x_{14}=\text{p27}$, $x_{15}=\text{pRB}$, $x_{16}=\text{ErbB2\_3}$, $x_{17}=\text{CDK6}$, $x_{18}=\text{CDK2}$, and $x_{19}=\text{MEK1}$. 

\smallskip

For any value of $m$, we translate the Boolean network to our multivalued setting by changing $\vee\rightsquigarrow \oplus$, $\wedge\rightsquigarrow \odot$ and $\lnot\rightsquigarrow \neg$:

\

 {\small
\noindent\begin{tabular}{r@{\hspace{1pt}}c@{\hspace{1pt}}l@{\hspace{2pt}}c@{\hspace{12pt}}r@{\hspace{1pt}}c@{\hspace{1pt}}l}
    $f_1 $& $=$& c &                  &$f_{11}$ &$=$ & $x_1 \odot x_9$ 
    \\
    $f_2$ &$=$ & $x_8   \oplus x_{19}$ &&$f_{12}$ &$=$&  $x_1 \odot x_6$   \\
    $f_3$ &$=$&$x_7       $           &&$f_{13}$ &$=$& $( x_8   \odot \neg( x_{16}  ) )  \oplus (  x_2   \odot \neg( x_{16}  ) ) $ 
    \\
    $f_4$ &$=$& $( x_{2} \odot  x_7  \odot x_{19}  )  \oplus ( x_2 \odot  x_7   \odot  x_8  ) $ &&$f_{14} $&$=$& $ x_2 \odot \neg( x_5 )  \odot \neg( x_7 ) \odot \neg( x_8)  \odot \neg( x_{18}) $
    \\
    $f_5$ &$=$&  $x_4   \odot \neg( x_{10}  )   \odot \neg( x_{14}  ) $  &&$f_{15}$ &$=$& $( x_5 \odot x_{17} \odot x_{18}     )  \oplus ( x_5 \odot x_{17} )$ 
    \\
    $f_6$ &$=$& $ c   $                &&$f_{16}$ &$=$&  $x_6 \odot  x_9$   \\
    $f_7$ &$=$& $ x_2\oplus x_8 \oplus x_{19} $ &&  $ f_{17}$ &$=$& $ x_4$  \\
    $f_8$ &$=$& $x_1  \oplus x_{11}  \oplus x_{12}   \oplus x_{13}  \oplus x_{16}$  &&  $f_{18}$ &$=$&  $x_3   \odot \neg( x_{10}  )  \odot \neg( x_{14}  )  $
    \\
    $f_9$ &$=$& c   &&    $f_{19}$ &$=$& $x_1   \oplus x_{11}  \oplus x_{12}  \oplus x_{13}  \oplus x_{16}.$ 
    \\
    $f_{10}$ &$=$&$ x_2  \odot \neg( x_5  )   \odot \neg( x_7  )  \odot \neg( x_8  )$ &&&& \\
 \end{tabular}
 }
 
 \
 
Following the reductions in \S~\ref{ssec:reductions}, we start by replacing $x_1=x_6=x_9=c$, $x_3=x_7$, $x_{17}=x_4$ and removing $f_1$, $f_3$, $f_6$, $f_9$ and $f_{17}$. As $x_{15}$ is not involved in any equation, we set
$x_{15}=(x_5 \odot x_{17} \odot x_{18})\oplus (x_5 \odot x_{17})$ and remove $f_{15}$.
 We then make further reductions  by replacing $x_{11}=x_{12}=x_{16}=c^2$, $x_{19}=x_8$ and removing $f_{11}$, $f_{12}$, $f_{16}$ and $f_{19}$. Moreover, as $x_{10}=x_2  \odot \neg( x_5  )   \odot \neg( x_7  )  \odot \neg( x_8  )$ we can replace this expression in $f_{14}$. The outcomes of these two series of reductions are shown in the two columns below:

\

 {\small
\noindent\begin{tabular}{r@{\hspace{1pt}}c@{\hspace{1pt}}l@{\hspace{2pt}}c@{\hspace{2pt}}r@{\hspace{1pt}}c@{\hspace{1pt}}l}
   $f_2$ & $=$ & $x_8   \oplus x_{19}$ & & $f_2$ & $=$ & $x_8   \oplus x_8$\\
   $f_4$ & $=$ & $( x_2 \odot  x_7  \odot x_{19}  )  \oplus ( x_2 \odot  x_7   \odot  x_8  )$ & & $f_4$ & $=$ & $( x_2 \odot  x_7  \odot x_8  )  \oplus ( x_2 \odot  x_7   \odot  x_8  )$
   \\
   $f_5$ & $=$ & $x_4   \odot \neg( x_{10}  )   \odot \neg( x_{14}  )$ & & $f_5$ & $=$ & $x_4   \odot \neg( x_{10}  )   \odot \neg( x_{14}  )$
   \\
   $f_7$ & $=$ & $x_2   \oplus x_8   \oplus  x_{19}$ & & $f_7$ & $=$ & $x_2   \oplus x_8   \oplus  x_8$
   \\
   $f_8$ & $=$ & $x_{11}  \oplus x_{12}    \oplus x_{13}  \oplus x_{16}\oplus c $ & $\rightsquigarrow$ & $f_8$ & $=$ & $x_{13}  \oplus c   \oplus c^2  \oplus c^2  \oplus c^2$
   \\
   $f_{10}$ & $=$ & $x_2  \odot \neg( x_5  )   \odot \neg( x_7  )  \odot \neg( x_8  )$ & & $f_{10}$ & $=$ & $x_2  \odot \neg( x_5  )   \odot \neg( x_7  )  \odot \neg( x_8  )$
   \\
   $f_{11}$ & $=$ & $c^2$ & & $f_{13}$ & $=$ & $( x_8   \odot \neg( c^2  ) )  \oplus (  x_2   \odot \neg( c^2  ) )$
   \\
   $f_{12}$ & $=$ & $c^2$ & & $f_{14}$ & $=$ & ${x_{10}} \odot \neg( x_{18} ) $
   \\
   $f_{13}$ & $=$ & $( x_8   \odot \neg( x_{16}  ) )  \oplus (  x_2   \odot \neg( x_{16}  ) )$ & & $f_{18}$ & $=$ & $x_7   \odot \neg( x_{10}  )  \odot \neg( x_{14}  ).$
   \\
   $f_{14}$ & $=$ & ${x_{10}} \odot \neg( x_{18} ) $ & & &  & 
   \\
   $f_{16}$ & $=$ & $c^2$ & &  & & 
   \\
   $f_{18}$ & $=$ & $x_7   \odot \neg( x_{10}  )  \odot \neg( x_{14}  )$ & &  &  &
   \\
   $f_{19}$ & $=$ & $ x_{11}  \oplus x_{12}  \oplus x_{13}  \oplus x_{16}   \oplus c$ & & &  & \\
 \end{tabular}
 }

\
 
 We  get the following $\odot$-$\neg$ network function, where we call $c'=c   \oplus c^2  \oplus c^2  \oplus c^2$ and $c''=\neg(c^2)$:
 
\
 
 {\small
\noindent\begin{tabular}{r@{\hspace{1pt}}c@{\hspace{1pt}}l@{\hspace{2pt}}c@{\hspace{2pt}}r@{\hspace{1pt}}c@{\hspace{1pt}}l}
    $\overline{f}_2$ &$=$ & $\neg(u_1)$  & &    $\overline{f}_{u_1}$ &$=$& $\neg(x_8)^2$\\
    $\overline{f}_4$ &$=$ & $\neg(u_3) $&&    $\overline{f}_{u_2}$ &$= $&$x_2 \odot   x_7    \odot  x_8$\\
    $\overline{f}_5$ &$=$ & $ x_4   \odot \neg( x_{10}  )   \odot \neg( x_{14}  ) $  && $\overline{f}_{u_3}$ &$=$ &$\neg(u_2)^2$\\
    $\overline{f}_7$ &$=$ & $ \neg(u_4)$   &&  $\overline{f}_{u_4}$ &$=$& $u_1 \odot \neg(x_2)$\\
    $\overline{f}_8$ &$=$ & $\neg(u_5)$&&    $ \overline{f}_{u_5}$ &$=$ &$\neg(x_{13})\odot \neg(c')$
    \\ 
    $\overline{f}_{10}$ &$=$ & $x_2  \odot \neg( x_5  )   \odot \neg( x_7  )  \odot \neg( x_8  ) $ &&    $\overline{f}_{u_6}$ &$=$&$ x_8   \odot ~c''$\\
    $\overline{f}_{13}$ &$=$ & $\neg(u_8) $ &&    $\overline{f}_{u_7}$ &$=$ &$x_2  \odot ~c''$\\
    $\overline{f}_{14}$ &$=$ & ${x_{10}} \odot \neg( x_{18} ) $  &&    $\overline{f}_{u_8}$ &$=$ &$\neg(u_6)\odot \neg(u_7)$.\\
    $\overline{f}_{18}$ &$=$ & $ x_7   \odot \neg( x_{10}  )  \odot \neg( x_{14}  )$  &&&&
 \end{tabular}
 }

  \

    Again following the reductions in \S~\ref{ssec:reductions}, we replace $x_2=\neg(u_1)$, $x_4=\neg(u_3)$, $x_7=\neg(u_4)$, $x_8=\neg(u_5)$, $x_{13}=\neg(u_8)$, and remove $\overline{f_2}$, $\overline{f_4}$, $\overline{f_7}$, $\overline{f_8}$, $\overline{f_{13}}$. We call $g_i$ the functions obtained after applying the reductions to $\overline{f_i}$.

\noindent \begin{minipage}{0.45\textwidth}
 {\small
 \begin{align*}
g_5 &=    \neg ( x_{10})   \odot \neg ( x_{14}  )   \odot \neg(u_3)  \\
g_{10} &=  \neg ( x_5  )  \odot  \neg(u_1)  \odot u_4   \odot u_5   \\
g_{14} &=  x_{10} \odot \neg( x_{18} )   \\
g_{18} &=  \neg ( x_{10}  )   \odot \neg ( x_{14}  )  \odot   \neg(u_4)  \\
g_{u_1} &= u_5^2\\
g_{u_2} &= \neg(u_1) \odot   \neg(u_4)    \odot  \neg(u_5)\\
g_{u_3} &= \neg(u_2)^2\\
g_{u_4} &= u_1^2\\
g_{u_5} &= u_8\odot \neg(c')\\
g_{u_6} &= \neg(u_5)   \odot ~c''\\
g_{u_7} &= \neg(u_1)  \odot ~c''\\
g_{u_8} &= \neg(u_6)\odot \neg(u_7).
\end{align*}
}
\end{minipage}
\begin{minipage}{0.45\textwidth}
\scalebox{0.8}
\small{
\[    \xymatrix{
    & & u_7\ar@{-|}[dr] & & u_6\ar@{-|}[dl] & C\ar@{->}@/_0.5pc/[l]_</0.4cm/{c''}\ar@{->}@/_2pc/[lll]_</1cm/{c''}\ar@{->}[ddl]^</1.3cm/{\neg(c')}\\
    & & & u_8\ar@{->}[dr] &  & \\
    x_{10}\ar@{-|}@/_1.5pc/[dddrr]\ar@{-|}[ddr]\ar@{->}[ddd] & & u_1\ar@{-|}@[red][ll]\ar@{-|}[uu]\ar@{->}@[blue][dd]_</1cm/{2}\ar@{-|}[ddrr] & & u_5\ar@{->}[ll]_</1cm/{2}\ar@{-|}[uu]\ar@{-|}[dd]  & \\
     & &  & &  & \\
    & x_{18}\ar@{-|}[dl] & u_4\ar@{-|}[l]\ar@{->}@/_0.7pc/@[blue][uull]\ar@{-|}[rr] & & u_2\ar@{-|}[d]^</0.4cm/{2} & \\
    x_{14}\ar@{-|}[ur]\ar@{-|}[rr] & & x_5\ar@{-|}@/_0.8pc/[uuull]  & & u_3\ar@{-|}[ll] & \\
   }
\]
 }
\end{minipage}

\
In order to further reduce the fixed point equations, we now apply the reduction in Proposition~\ref{prop:pos_neg_reduction} to node 10 with the aid of the arrows depicted with red and blue in the network digraph on the right. We then get $x_{10}=0$ and by replacing this information in $g_{14}$ we also get $x_{14}=0$. We make the corresponding replacements and then remove $g_{10}$ and $g_{14}$. We further replace $x_5=\neg(u_3)$ and $x_{18}=\neg(u_4)$, and remove $g_5$ and $g_{18}$. As $u_3$ is not involved in any of the remaining equations, we set $u_3=\neg(u_2)^2$ and remove $g_{u_3}$. With the same reasoning, we replace $u_2=\neg(u_1) \odot   \neg(u_4)    \odot  \neg(u_5)$ and remove $g_{u_2}$, and $u_4=u_1^2$ and remove $g_{u_4}$. We can moreover replace $u_8=\neg(u_6)\odot \neg(u_7)$ and remove $g_{u_8}$. Call $h_i$ the functions obtained after applying the proposed reductions:

\noindent \begin{minipage}{0.47\textwidth}
\begin{align} \label{eq:example}
\nonumber h_{u_1} &= u_5^2\\
h_{u_5} &= \neg(u_6)\odot \neg(u_7)\odot \neg(c')\\
\nonumber h_{u_6} &= \neg(u_5)   \odot ~c''\\
\nonumber h_{u_7} &= \neg(u_1)  \odot ~c''
\end{align}
\end{minipage}
\begin{minipage}{0.47\textwidth}
{\small
\[
    \xymatrix{
     u_7\ar@{-|}[dr] & u_6\ar@{-|}@<-.4ex>[d] & C\ar@{->}@/_0.5pc/[l]_</0.4cm/{c''}\ar@{->}@/_2pc/[ll]_</1cm/{c''}\ar@{->}[dl]^</0.7cm/{\neg(c')}\\
    u_1\ar@{-|}[u] & u_5\ar@{->}[l]_</0.5cm/{2}\ar@{-|}@<-.4ex>[u] & \\
    }
\]
}
\end{minipage}
\vskip5mm

We will refer to system~\eqref{eq:example} as the {\em reduced system}.
\bigskip

So far we have $x_1=x_6=x_9=c$, $x_2=\neg(u_1)$, $x_4=x_5=x_{17}=\neg(u_3)$, $x_3=x_7=x_{18}=\neg(u_4)$, $x_8=x_{19}=\neg(u_5)$, $x_{10}=x_{14}=0$, $x_{11}=x_{12}=x_{16}=c^2$, $x_{13}=\neg(u_8)$, $x_{15}=(x_{18} \odot x_5 \odot x_{17})\oplus (x_5 \odot x_{17})$, and $u_2=\neg(u_1) \odot   \neg(u_4)    \odot  \neg(u_5)$, $u_3=\neg(u_2)^2$, $u_4=u_1^2$, and $u_8=\neg(u_6)\odot \neg(u_7)$. 

\bigskip

Note that $c^2=\max\{0,2c-1\}$. Then,

\medskip

\begin{itemize}\itemsep=12pt
 \item If $c\leq\frac{1}{2}$, then $c^2=0$ and $c'=c$, $c''=1$. We obtain then $u_6=\neg(u_5)$, $u_7=\neg(u_1)$ and we can remove $h_{u_6}$ and $h_{u_7}$ and get
    $h_{u_1} = u_5^2$, 
$h_{u_5} = u_5\odot u_1\odot \neg(c)$.
  
By going one step further and replacing $u_1=u_5^2$ we are left with only one fixed point equation:
\[
 u_5=u_5^3\odot \neg(c).
\]
\begin{itemize}
 \item If $u_5\leq \frac{2}{3}$, then $u_5^3=0$ an then necessarily $u_5=0$. This gives the following fixed point:
\[
      x_1=x_6=x_9=c, \; x_{10}=x_{11}=x_{12}=x_{14}=x_{16}=0, 
    \]
    \[
      x_2=x_3=x_4=x_5=x_7=x_8=x_{13}=x_{15}=x_{17}=x_{18}=x_{19}=1.
    \]
 \item If $u_5>\frac{2}{3}$ then
 $ u_5=\max\{0,3u_5+(1-c)-3\}=3u_5+(1-c)-3 $\newline
$ \Leftrightarrow \; 2u_5=2+c  \; \Leftrightarrow \; u_5=1+\frac{c}{2}.$
This can only happen if $u_5=1$ and $c=0$, and then, $c'=0$, $c''=1$, and by substitution we obtain
     $x_i=0$ for $i\in \{1,\dots, 19\}$.
 
\end{itemize}

 \item If $c> \frac{1}{2}$, then $c^2= 2c-1$ and $c'=\min\{1,7c-3\}$. 
  \begin{itemize}\itemsep=12pt
   \item If moreover $c\geq\frac{4}{7}$, then $c'=1$, $\neg(c')=0$ and this gives $u_5=0$.  We deduce then $u_1=0$, $u_6=u_7=c''=\neg(c^2)=2-2c$ and  $u_8=c^4=\max\{0,4c-3\}$. By replacing these values we obtain the unique fixed point:
    \[
      x_1=x_6=x_9=c, \; x_{11}=x_{12}=x_{16}=2c-1, \; x_{13}=\neg(c^4), \; x_{10}=x_{14}=0,
    \]
    \[
      x_2=x_3=x_4=x_5=x_7=x_8=x_{15}=x_{17}=x_{18}=x_{19}=1.
    \]
   \item If $c\in (\tfrac{1}{2},\tfrac{4}{7})$, then $0<c^2=2c-1<\frac{1}{7}$ and  $c''=\neg(c^2)=2(1-c)$. It is easy to check that $u_1=u_5=0$, $u_6=u_7=2(1-c)$ is a fixed point of the reduced system and, by replacing, we obtain:
    \[
      x_1=x_6=x_9=c, \; x_{11}=x_{12}=x_{16}=2c-1, \; x_{10}=x_{14}=0,
    \]
    \[
      x_2=x_3=x_4=x_5=x_7=x_8=x_{13}=x_{15}=x_{17}=x_{18}=x_{19}=1.
    \]
  \end{itemize}
  Indeed, we checked that is the unique fixed point in two cases. With $m=9$ and $c=5/9$ (and so $c'=c''=8/9$), and with $m=13$ and $c=7/13$ (and $c'=10/13, c''=12/13$), using the implementation by A. Galarza Rial~ \cite{Aye}. In this not so big example, these computations are heavy to be made by hand.
\end{itemize}

We will explain in the next section how to systematically compute the fixed points.

\section{Computing fixed points}\label{sec:sstates}

Given $X_m=\left\{0,\frac{1}{m},\dots,\frac{m-1}{m},1\right\}$ with $m\geq 1$ arbitrary, and a multivalued network $F$ over $X_m$, 
we will now focus on effectively finding the fixed points of the system. Our first step is presented in Theorem~\ref{thm:ell} that gives a unique representation of a $\odot$-$\neg$ function in terms of linear forms. Algorithm~\ref{algo2} is the basis for our implementation. We show in Lemma~\ref{lem:algoBool} how it is simplified in the Boolean case. Theorem~\ref{th:regions} gives conditions for the fixed points to be easily computable.

 \begin{theorem}\label{thm:ell}
Given a $\odot$-$\neg$ network function $F=(f_1,\dots,f_n): X_m^n \to X_m^n$,
any nonzero coordinate function $f_i$
   can be written in a unique way as follows:
 \begin{equation} \label{eq:fimax}f_i= \max\{0,\ell_i(x)\},
 \end{equation}
 where $\ell_i$ are linear functions of the form:
 \begin{align}
 \label{eq:ell}  \ell_i(x) &= \ell_i'(x)+\frac{c_i^*}m, \quad \frac{c_i^*}m\in \mathbb Q_{\leq 1},\\
 \nonumber \ell_i'(x) &=\sum_{j\in A_i} p_{ij} x_j - \sum_{j\in B_i} q_{ij} x_j, \quad \text{and} \quad \frac{c_i^*}m = \frac{c_i}m -\sum_{j\in A_i} p_{ij},
 \end{align}
with $\frac {c_i} m \in X_m$, $p_{ij}, q_{ij}$ integers
in $\{1,\dots,c_i\}$ and $A_i \cap B_i = \emptyset$.

\end{theorem}

\begin{proof}
We know from Proposition~\ref{prop:c} and Theorem~\ref{th:unique} that any nonzero $\odot$-$\neg$ function $f_i$ can be uniquely written as
\begin{equation} \label{eq:AiBi}
      f_i=\underset{j\in A_i}{\bigodot}x_j^{p_{ij}}\odot \underset{j\in B_i}{\bigodot} (\neg(x_j))^{q_{ij}}\odot \frac{c_i}m, \; A_i,B_i\subseteq\{1,\dots,n\},  \; A_i\cap B_i=\emptyset,
 \end{equation}
with $\frac {c_i} m \in X_m$, $p_{ij}, q_{ij}$ integers
in $\{1,\dots,c_i\}$ and $A_i \cap B_i = \emptyset$.

Now, from item~\ref{it:several_products} in Proposition~\ref{prop:oper} we deduce Equation~\eqref{eq:fimax} with 
\begin{align*}
\ell_i(x)&=\sum_{j\in A_i} p_{ij}x_j+\sum_{j\in B_i}q_{ij}(1-x_j)+\frac{c_i}m-\left(\sum_{j\in A_i} p_{ij}+\sum_{j\in B_i}q_{ij}+1-1\right)\\
&=\sum_{j\in A_i} p_{ij}x_j-\sum_{j\in B_i}q_{ij}x_j+ \frac{c_i}m -\sum_{j\in A_i} p_{ij}.
\end{align*}

As $\frac{c_i}m\in X_m$ and $\sum_{j\in A_i} p_{ij}\in\Z_{\geq 0}$ we deduce that $\frac{c_i^*}m\in \mathbb Q_{\leq 1}$, which is what we wanted to prove.
\end{proof}

Reciprocally:

\begin{lemma}
    Any linear function of the form $\ell=\sum_{j\in A} p_{j} x_j - \sum_{j\in B} q_{j} x_j+(\frac c m-\sum_{j\in A} p_{j})$, with $p_j, q_j \in \Z_{\ge 0}$ and $\frac c m \in X_m$,
         defines a  $\odot$-$\neg$ function $f$.
\end{lemma}  

\begin{proof}
   It is enough to consider the function  $f=\underset{j\in A}{\bigodot}x_j^{p_{j}}\odot \underset{j\in B}{\bigodot} (\neg(x_j))^{q_{j}}\odot \frac c m$.
\end{proof}

We deduce from Theorem~\ref{thm:ell} the validity of Algorithm~\ref{algo2} below to compute the fixed points of any ($\odot$-$\neg$)-network function $F:X_m^n \to X_m^n$.

First, let $J$ be the set of  indices $i$ for which $f_i= \ell_i$ consists of a constant, a single variable or its negation. Of course for each of these linear functions we can write them as $\max\{0,\ell_i(x)\}$, but this would introduce unnecessary  branchings in the computations. 
When $J$ is non-empty, we can apply reduction~\ref{it:xi_negxi} in \S~\ref{ssec:reductions} and eliminate 
the variables with indices in $J$.
Note that if we start with a $\odot$-$\neg$ function, we obtain a reduced function $f_{red} :X_m^{n-|J|} \to X^{n-|J|}$ that is also of this form. Then, the number of regions to consider in Step~1 of the following Algorithm~\ref{algo2} is $2^{n-|J|}$.
The linear systems to be solved in Step~2 of the Algorithm have also at most $n-|J|$ variables. In Step~3, we also add the values of the variables with indices in $J$ to output the fixed points.
To alleviate the notation, we assume that $J = \emptyset$ but the input of Algorithm~\ref{algo2} should be the reduced function $f_{red}$ and not $f$.

\begin{algorithm}
\caption{We compute the fixed points of a $\odot$-$\neg$ network with at least one $\odot$ operation in each $f_i$. }\label{algo2}
  \flushleft{\textbf{Input:} A $\odot$-$\neg$ network $F = (f_1, \dots, f_n): X_m^n \to X_m^n$, with $f_i = \max\{0,\ell_i(x)\}$ as in~\eqref{eq:fimax}.}\\
  \textbf{Output:} The fixed points of $F$.
  \begin{algorithmic}
    \State \textbf{Step 1:} For each $I\subseteq \{1,\dots,n\}$ consider the  region $\mathcal{R}_I=\left\lbrace \begin{array}{c} \ell_i(x)\geq 0 \; i\in I\\ \ell_j(x)\leq 0 \; j\notin I\end{array}\right\rbrace$. 
    \State \textbf{Step 2:} Solve the linear system $\left\lbrace \begin{array}{rl} \ell_i(x)=x_i & i\in I \\ 0=x_j & j\notin I\end{array}\right.$. In case there is a single (necessarily rational) solution, check if $x \in X_m^n$ and then check if  $\ell_j(x)\leq 0$ for all $j\notin I$. If this is satisfied, then $x$ is a fixed point of $F$ in $\mathcal{R}_I$.  
      \State \textbf{Step 3:} 
    In case the linear system has infinitely many rational solutions, change the variables to $y_i=m x_i$: multiply by $m$ the constants $\frac{c_i}m$ to translate the linear forms $\ell_i-x_i, \ell_j$  to integer linear forms $\bar{\ell}_i(y)-y_i, \bar{\ell}_j$ and find the lattice points in the rational polytope $$P_I= \{ \bar{\ell}_i(y)-y_i=0, i \in I, y_j=0, j \notin I, 0\le y_i \le m, i \in I, \bar{\ell}_j(y)\leq 0, j\notin I \}.$$ 
        Translate back the solutions to solutions $x \in X_m^n$.
    \end{algorithmic}
\end{algorithm}

Algorithm~\ref{algo2} is greatly simplified in the Boolean case  $m=1$. Computing the fixed points is restricted to the verification of the set inclusions/emptiness described in the following lemma,  which is straightforward to prove. Again, as in the general case, we assume reduction~\ref{it:xi_negxi}in  \S~\ref{ssec:reductions} has been applied and that the corresponding variables have been eliminated.

 \begin{lemma} \label{lem:algoBool}
Let  $F = (f_1, \dots, f_n): X_1^n \to X_1^n$ be a $\odot$-$\neg$ (AND-NOT) Boolean network function with $f_i=\underset{j\in A_i}{\bigodot}x_j\odot \underset{j\in B_i}{\bigodot} \neg(x_j)$,   with $A_i\cap B_i=\emptyset$ disjoint subsets of  $\{1, \dots, n\}$. Given $x \in X_m^n$, set $A(x),B(x)\subseteq \{1,\dots,n\}$, $A(x) = \{ j~:~x_j=1\}$, $B(x) = \{ j~:~x_j=0\}$. 
  
Then, the following statements are equivalent:
  \begin{enumerate}
  \item $x$ is a fixed point of $F$.
  \item \label{it:Boolean_Ax} For any index $i$, 
  \begin{itemize}
      \item if $i \in A(x)$  it holds that $A_i \subset A(x)$ and $B_i \cap A(x) = \emptyset$,
      \item if $i \in B(x)$ it holds that $B(x) \cap A_i \not= \emptyset$ or $A(x) \cap B_i \not= \emptyset$.   
  \end{itemize} 
  \end{enumerate}

 \end{lemma}

Now, going back to Algorithm~\ref{algo2}, how does one algorithmically find all the lattice points in rational polytopes? 
A good implementation of Barvinok's algorithm mentioned in the Introduction was provided in the free software Latte~\cite{Latte}, explained in the paper~\cite{DeLoera} (now available in SageMath). We used the free implementation in~\cite{BarvinokSoftware}, explained in the report~\cite{BarvinokLeuven} that counts the lattice points and also provides them explicitly.

Note that for any $m\geq 1$ there are at most $2^{n-|J|}$ regions $\mathcal{R}_I$ (independently of $m$), while the exhaustive search of fixed points requires $(m+1)^{n}$ initial states. Moreover, the computations are parallelizable for different $I$ and further simplifications and improvements in the tree of computations might be certainly done in particular instances.

The best possible case is when the square linear systems in Step~2 have a unique solution (necessarily rational). In this case, solving each system has complexity of order $O(|I|^3)$  and we then check if its solution lies in $X_m^n$. We give in Theorem~\ref{th:regions} below an easy general condition under which this happens. 
We remark however that when the rank is not maximum, the complexity of Step~2 might raise (up to the order of $(m+1)^n$).

\begin{definition} \label{def:Mij}
Consider an $\odot$-$\neg$ network function $F=(f_1, \dots, f_n)$ with at least one $\odot$ operation in each $f_i$, let $\ell_1, \dots, \ell_n$ the associated
linear functions as in the statement of Theorem~\ref{thm:ell}.
Given $I\subseteq \{1,\dots, n\}$ of cardinality $s$ we denote by
$M_I$ the $s \times s$ matrix  with both rows and columns indicated by $I$ and with
entries: 
\begin{equation}\label{eq:Matrix_MI}
 (M_I)_{ij}=\left\lbrace\begin{array}{ll}\ell_{ii}-1 & \text{if } i=j,\\
\ell_{ij} & \text{if } i\neq j
\end{array}\right.
\end{equation}
\end{definition}

\begin{theorem}\label{th:regions}
 Given a $\odot$-$\neg$ network defined by $(f_1, \dots, f_n)$.  Let $J \subset \{1, \dots, n\}$ denote the
set of indices $i$ for which $f_i$ has no $\odot$ operations.
If for $I\subseteq \{1,\dots,n\}$ with $I \subset J^c$, the matrix $M_I$ defined in~\eqref{eq:Matrix_MI} satisfies 
$\det(M_I)\neq 0$, then there is at most one fixed point in the region $\mathcal{R}_I$. 

\end{theorem}

The proof of Theorem~\ref{th:regions} is immediate as $M_I$ is the matrix of the linear
system $\left\lbrace \begin{array}{rl} \ell_i(x)=x_i & i\in I \\ 0=x_j & j\notin I\end{array}\right.$. The following corollary is also immediate.

\begin{corollary}
 Let $L'\in \Z^{n\times n}$ be the matrix of coefficients of the linear forms $\ell'_1,\dots,\ell'_n$ as in Theorem \ref{thm:ell} and  $Id_n$ be the $n \times n$ identity matrix. If $\det(L'-Id_n)\neq 0$ then there exists at most a single fixed point with all nonzero coordinates.
\end{corollary}

We end this section showing how to compute the fixed points using our app~\cite{Aye}.

\begin{example} 
 We consider again the biological example in \S~\ref{ssec:mammalian},  in the case of $m=9$, $c=5/9, c'=  c''=8/9$. The reduced system~\eqref{eq:example} has four variables (ordered $u_1, u_5, u_6, u_7$). It can be written as:
 
\begin{align*} \label{eq:example2}
\nonumber h_{u_1} &= \max\{0, 2 u_5 -1\} \\
h_{u_5} &= \max \{0, -u_6 -u_7 + 1/9\}\\
\nonumber h_{u_6} &= \max\{0, -u_5+8/9\}\\
\nonumber h_{u_7} &= \max\{0, -u_1+8/9\}
\end{align*} 
\end{example}

After installing our app {\sl Fixed points of MV networks} from the github repository~\cite{Aye}, input the coefficients of the linear forms in each function, multiplying {\em only} the constant terms
by $m$, to get the output as in Figure~\ref{fig:app}. Of course, there is a lot of room to improve this implementation.

\begin{figure}[ht!]
\begin{center}
\includegraphics[scale=0.5,trim={0 0 0 2mm},clip]{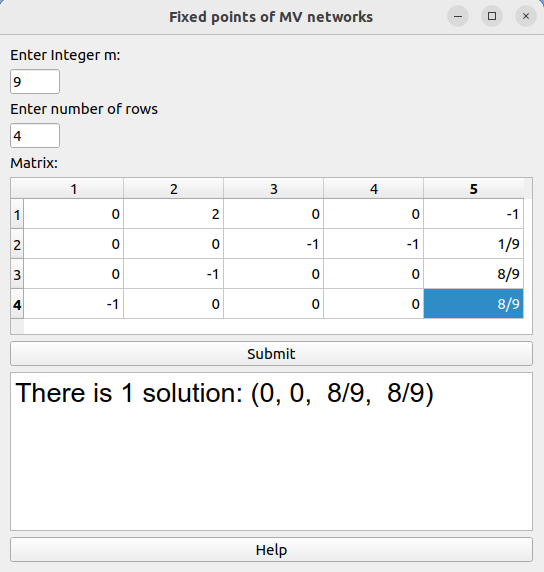}
\caption{Screenshot of the app~\cite{Aye}}\label{fig:app}
\end{center}
\end{figure}

\section{Application to the Model for Denitrification in Pseudomonas aeruginosa} \label{sec:bioexample}

Denitrification is a process mediated by bacteria that converts nitrogen compounds into gaseous forms. It is a final step in the nitrogen cycle and is important for recycling nitrogen. In \cite{ABL} the authors try to unveil the environmental factors that contribute to the greenhouse gas N$_2$O accumulation, especially in Lake Erie (Laurentian Great Lakes, USA). Their model predicts the long-term behavior of the denitrification pathway and involves molecules (O$_2$, PO$_4$, NO$_3$, $NO_2$, $NO$, $N_{2}O$, $N_2$), proteins ($PhoRB$, $PhoPQ$, $PmrA$, $Anr$, $Dnr$, $NarXL$, $NirQ$) and genes($nar$, $nir$, $nor$, $nos$). The process in P. aeruginosa is described by  three Boolean external parameters (O$_2$, PO$_4$, NO$_3$),  four Boolean variables ($PhoRB, PhoPQ, PmrA$,$Anr$), and eleven ternary variables ($NarXL$, $Dnr$, $NirQ$, $nar$, $nir$, $nor$, $nos$, $NO_2$, $NO$, $N_{2}O$, $N_2$). In our setting, we assume that they take three values (in $X_2$). The corresponding update rules  in \cite{ABL} are:

\begin{center}
    \begin{tabular}{c@{\hspace{1mm}}c@{\hspace{1mm}}lc@{\hspace{1mm}}c@{\hspace{1mm}}lc@{\hspace{1mm}}c@{\hspace{1mm}}l}
    $f_{PhoRB}$ & $=$ & $\neg(PO_4)$, & $f_{NarXL}$ & $=$ & $\min\{Anr,NO_3\}$, & $f_{NO}$ & $=$ & $\min\{NO_2,nir\}$,\\
    $f_{PhoPQ}$ & $=$ & $PhoRB$, & $f_{nir}$ & $=$ & $\max\{NirQ,\min\{Dnr,NO\}\}$, & $f_{N_2O}$ & $=$  & $\min\{NO,nor\}$,\\
    $f_{PmrA}$ & $=$ & $\neg(PhoPQ)$, & $f_{nor}$ & $=$ & $\max\{NirQ,\min\{Dnr,NO\}\}$, & $f_{N_2}$ & $=$ & $\min\{N_2O,nos\}$,\\
    $f_{Anr}$ & $=$ & $\neg(O_2)$, & $f_{nos}$ & $=$  & $\min\{Dnr,NO\}$. & & &  
    \end{tabular}
\end{center}

As we pointed out in our Introduction, they cannot accurately express four variables (${Dnr}$, ${NirQ}$, ${nar}$ and ${NO_2}$) using {\em only} the operators MAX, MIN and NOT (indicated with a $*$ in~\cite[Table~3]{ABL}),  and are then forced to show all the possible values in supplementary tables. 
 From the transition tables of each of the 15 variables, they build a polynomial dynamical system and compute the fixed points of the system (with a software based on Gr\"obner basis computations) under the environmental conditions of interest that emerge after fixing the values of the external parameters. 
 
\

To further show the strength of the tools we developed in this paper, we implemented their model in our setting. We can translate the functions shown above to operations with $\odot$ and $\oplus$ using Proposition~\ref{prop:oper}. For the remaining functions, $f_{Dnr}$, $f_{NirQ}$, $f_{nar}$ and $f_{NO_2}$, we considered the supplementary tables in \cite{ABL} without the smoothing process since the fixed points are not changed (see Lemma~\ref{lem:continuity}). 
For instance, the values for $NirQ$ are shown in Figure~\ref{fig:NirQ}. We can see that the fourth transition $(NarXL=1,Dnr=0)\mapsto NirQ=\frac{1}{2}$ does not agree with  the function $\max\{NarXL,Dnr\}$ in~\cite{ABL}. However, we can represent any multivalued function in our framework. There is always a standard alternative to do this by building the interpolation function in Theorem~\ref{th:logic}. In this case, there is a much better alternative: find a more intuitive and simple expression where $Dnr$ is an activator and $NarXL$ is a mild activator (see Motifs~1 and~4 in \S~\ref{ssec:motifs}). See Figure~\ref{fig:NirQ}.
\begin{figure}[ht!]
 \begin{minipage}{0.25\textwidth}
  \begin{tblr}{
    colspec = {|c|c|c|},
    row{1} = {purple7},
    column{3} = {teal7},
    row{5} = {yellow},
  }
    \hline
      NarXL & Dnr & NirQ \\
    \hline
        0 & 0 & 0\\
    \hline
        0 & 1/2 & 1/2\\
    \hline
        0 & 1 & 1\\
    \hline
        1 & 0 & 1/2\\
    \hline
        1 & 1/2 & 1\\
    \hline
        1 & 1 & 1\\
    \hline
  \end{tblr}
 \end{minipage}
\begin{minipage}{0.7\textwidth}
  Update rule in~\cite{ABL}: $*\max\{NarXL,Dnr\}$. The value $\tfrac{1}{2}$ at $(0,1)$ in the fourth row is not obtained with this expression.

  Interpolation expression: 
  {\small
  \begin{align*}
        f_{NirQ} &= \tfrac{1}{2}\odot \neg(NarXL)^2 \odot \neg\left(\neg(Dnr) \odot \tfrac{1}{2}\right)^2 \odot \neg(Dnr \odot \tfrac{1}{2})^2\\
        &\oplus \neg(NarXL)^2 \odot Dnr^2\\
        &\oplus \tfrac{1}{2} \odot NarXL^2 \odot \neg(Dnr)^2\\
        &\oplus NarXL^2 \odot \neg\left(\neg(Dnr) \odot \tfrac{1}{2}\right)^2 \odot \neg\left(Dnr \odot \tfrac{1}{2}\right)^2\\
        &\oplus NarXL^2 \odot Dnr^2.
  \end{align*}
  }
  
  Simpler expression: $f_{NirQ} = (NarXL \ominus \tfrac{1}{2}) \oplus Dnr$.

 \end{minipage}
 \caption{$NirQ$ update rule. The table on the left is the multivalued non-smooth version of S3~Table in~\cite{ABL}.  On the right: the update rule approximation shown in~\cite{ABL}; the update rule obtained by interpolation from the table according to Theorem~\ref{th:logic}; and the expression obtained by considering $Dnr$ as an activator and $NarXL$ as a mild activator. The last two expressions coincide with the values in the table.}\label{fig:NirQ}
\end{figure}

We can represent $f_{NirQ} = (NarXL \ominus \frac{1}{2}) \oplus Dnr$, $f_{Dnr} = (Anr\ominus \frac{1}{2}) \oplus (\neg(PmrA)\odot Anr \odot NarXL)$, and $f_{NO_2} =NO_3  \odot nar$ which have intuitive interpretations.

In our setting, we  fix the value of $m$ for all functions in the network. In this model, most variables take three values, corresponding to the choice $m=2$. 
Their Boolean variables can be reduced to three. Instead of assigning them arbitrarily values in $X_m = \{0,\tfrac{1}{2},1\}$, for instance, the values $0$ or $1$,  it is more reasonable in our setting to consider  the $8$ different resulting models in the reminder $3$-valued variables, for each of the possible  values of the Boolean variables.

 We then set the values of each Boolean external parameter (O$_2$, PO$_4$, NO$_3$) and study the corresponding fixed points of each of the eight possible systems separately. We expand on the details below. The two relevant conditions for denitrification with low O$_2$ are: O$_2=~$PO$_4=0$, NO$_3=1$ (considered ``the perfect condition for denitrification'') and  O$_2=0$, PO$_4=~$NO$_3=1$ (the denitrification condition disrupted by PO$_4$ availability). There are other two conditions with low O$_2$ that the authors in~\cite{ABL} disregard because they are less likely in fresh-waters: O$_2=~$PO$_4=~$NO$_3=0$, and O$_2=~$NO$_3=0$, PO$_4=1$, but nonetheless we computed the fixed points in our setting since they are easy to find. We also address the remaining conditions, with high O$_2$ (the ``aerobic conditions").  The fixed points we computed in all cases coincide with those found in~\cite{ABL}.
As remarked before, it is not necessary to implement the continuous versions of the update functions since the fixed points are not changed.

As we mentioned, we translated the MIN/MAX functions in~\cite[Table~3]{ABL} using Proposition~\ref{prop:oper}. For those variables the authors could not explain with MIN/MAX expressions, we represent exactly the update function presented in the supplementary tables in~\cite{ABL} (without the smoothing process) with the following expressions over $X_2$:
\begin{align*}
  f_{NirQ} &= (NarXL \ominus \tfrac{1}{2}) \oplus Dnr,\\
  f_{Dnr} &= (Anr\ominus \tfrac{1}{2}) \oplus (\neg(PmrA)\odot Anr \odot NarXL),\\
  f_{NO_2}&=NO_3  \odot nar, \\
  f_{nar} &= \left[\left(NarXL\oplus \tfrac{1}{2}\right)^2 \odot \left(Dnr \oplus \tfrac{1}{2}\right)^2 \odot \left(NO\oplus \tfrac{1}{2}\right)^2\right] \oplus \\
  &\oplus \left[ \tfrac{1}{2} \odot \left(NarXL \oplus \left(Dnr^2 \odot NO^2\right) \right) \right].\\   
\end{align*}

We then set the values of each Boolean external parameter (O$_2$, PO$_4$, NO$_3$) and obtain eight different systems. For each of them we compute the fixed points. The condition with low O$_2$, low PO$_4$ and high NO$_3$ (O$_2=$PO$_4=0$, NO$_3=1$) is considered in~\cite{ABL} ``the perfect condition for denitrification''. The condition with low O$_2$, high PO$_4$ and high NO$_3$ (i.e. O$_2=0$, PO$_4=$NO$_3=1$) corresponds to the denitrification condition disrupted by PO$_4$ availability. There are two conditions that they disregard because they are less likely in fresh-waters: low O$_2$, low PO$_4$ and low NO$_3$ (O$_2=$PO$_4=$NO$_3=0$), and low O$_2$, high PO$_4$ and low NO$_3$ (O$_2$=NO$_3=0$, PO$_4=1$), but as the fixed points can be computed easily in our setting (and are unique in each case) we show them in Table~\ref{tab:fp}. The remaining conditions are considered as ``aerobic conditions" and their fixed points (which are unique for each case) are also immediate to obtain and are shown in Table~\ref{tab:fp}. 

\begin{table}[ht!]
{\scriptsize
\begin{tabular}{ |c|c|c|c|c|c|c|c|c|c|c|c|c|c|c|c|c|c|}
\hline
\multicolumn{3}{|c|}{\small External parameters} & \multicolumn{15}{|c|}{\small Steady State}\\
\hline
 $O_2$ & $PO_4$ & $NO_3$ & PhoRB & PhoPQ & PmrA & Anr & NarXL & Dnr & NirQ & nar & nir & nor & nos & $NO_2$& NO & $N_2O$ & $N_2$\\
 \hline
 0 & 0 & 0 & 1 & 1 & 0 & 1 & 0 & $\frac{1}{2}$ &  $\frac{1}{2}$ & 0 &  $\frac{1}{2}$ &  $\frac{1}{2}$ & 0 &0 & 0 & 0 & 0\\
\hline
0 & 1 & 0 & 0 & 0 & 1 & 1 &  0 &  $\frac{1}{2}$ &  $\frac{1}{2}$ &  0 &  $\frac{1}{2}$ &   $\frac{1}{2}$ &  0 & 0 & 0 & 0 &  0 \\
\hline
1 & 0 & 0 & 1 &  1 & 0 & 0 & 0 & 0 &  0 &  0 &  0 & 0 &  0 &  0 & 0 &  0 &  0\\
\hline
 1 & 0 & 1 & 1 &  1 & 0 &  0 &  0 & 0 & 0 & 0 & 0 & 0 & 0 &  0 & 0 & 0 &  0\\
\hline
1 & 1 &  0 &  0 &  0 & 1 & 0 & 0 & 0 & 0 &  0 &  0 & 0 & 0 & 0 & 0 & 0 & 0\\
\hline
1 &1 & 1 &  0 &  0 &  1 &  0 &  0 &  0 &  0 & 0 &  0 & 0 & 0 &  0 &  0 & 0 & 0\\
\hline
\end{tabular}
}
\caption{Steady states under the environmental conditions determined by the external parameters. The first two conditions are disregarded in~\cite{ABL} because they are not  biologically relevant. The last four cases (the "aerobic conditions") coincide with~\cite[Fig.~2]{ABL}.}\label{tab:fp}
\end{table}

For the two cases of interest O$_2=0$, PO$_4=0$, NO$_3=1$ (``the perfect condition for denitrification'') and O$_2=0$, PO$_4=1$, NO$_3=1$ (``the denitrification condition disrupted by PO$_4$ availability") we built the $\odot$-$\neg$ system in Theorem~\ref{thm:DN} by adding $20$ new variables. We apply the reductions in \S~\ref{ssec:reductions} and obtain, in each case, reduced systems for the fixed points. 

The case O$_2=0$, PO$_4=1$, NO$_3=1$ renders a system of three equations in three variables which has a unique solution and determines the fixed point:

\

{\small
\begin{tabular}{ |c|c|c|c|c|c|c|c|c|c|c|c|c|c|c|c|}
\hline
PhoRB & PhoPQ & PmrA & Anr & NarXL & Dnr & NirQ & nar & nir & nor & nos & $NO_2$& NO & $N_2O$ & $N_2$\\
\hline
1&1&0&1&1&1&1&1&1&1&1&1&1&1&1\\
\hline
\end{tabular}
}

\

The case O$_2=0$, PO$_4=0$, NO$_3=1$ gives rise to a system of five equations in five variables which has a unique solution and determines the fixed point:

\

{\small
\begin{tabular}{ |c|c|c|c|c|c|c|c|c|c|c|c|c|c|c|c|}
\hline
PhoRB & PhoPQ & PmrA & Anr & NarXL & Dnr & NirQ & nar & nir & nor & nos & $NO_2$& NO & $N_2O$ & $N_2$\\
\hline
0&0&1&1&1&$\frac{1}{2}$&1&1&1&1&$\frac{1}{2}$&1&1&1&$\frac{1}{2}$\\
\hline
\end{tabular}
}

\

All the fixed points we obtain coincide with the fixed points obtained in~\cite{ABL}. We have also implemented their model using the indicator functions in Theorem~\ref{th:logic} for $f_{Dnr}$, $f_{NirQ}$, $f_{nar}$ and $f_{NO_2}$, and computed the fixed points of each of the eight systems determined by fixing the values of the external parameters. The steady states in Table~\ref{tab:fp} were obtained straightforwardly. We used the implementation~\cite{Aye} (without considering any network reductions) for the remaining two cases. 

\section{Discussion}
Discrete-time and discrete-space models, such as the multivalued (and in particular, Boolean)  networks studied here, are broadly applicable in biology and beyond, such as engineering, computer science, and physics. Mathematically, they are simply set functions on strings of a given finite length over a finite alphabet. 
Their dynamics is captured by a directed graph, exponential in the number of variables of the model. Traditionally, exhaustive simulation or sampling the state space, for larger models, was the only approach to characterize model dynamics. One strategy to make mathematical tools available for this purpose is based on the observation that any function $f: k^n \longrightarrow k$ can be expressed as a polynomial function over a finite field $k$. 
This allows one to use tools from computer algebra. For instance, primary decomposition of monomial ideals can be used for reverse-engineering of gene regulatory networks from time courses of experimental observations~\cite{Ve1,Ve3}. 

Here, we propose a different approach to bring mathematical tools to bear on the problem of analysis of a broad class of discrete models. Using multivalued logic, we connect this class of models to problems in combinatorics and enumerative geometry, in particular the calculation of rational points in polytopes. This is combined with a model reduction method that generally results in substantially smaller models with steady states  that bijectively correspond to the steady states of the original model. The reduction steps are ad hoc, and there is no theory behind them that would allow the derivation of a minimal reduced model or even a result about whether the order in which the reductions are applied matters in terms of the final reduced model. We believe that there is further interesting mathematics to be discovered in this context. The next important step is to use the tools we present to understand the full dynamics of the models.

Generally, we believe that multivalued networks and related models are rich and fascinating mathematical objects at the crossroads of combinatorics, algebra, and dynamical systems. They can be studied from the point of view of applications, or for purely mathematical reasons, or both, as in \cite{R23}. 

\section*{Acknowledgements} We are grateful to Alejandro Petrovich and to Patricio D\'{\i}az Varela for the references and explanations about MV-algebras in logic.  We are also indebted to the referees for their extensive and very helpful comments and references.

\end{document}